\documentclass[12pt,a4paper, leqno]{article}
\usepackage{amssymb, amscd, amsthm, amsmath,latexsym, amstext,mathrsfs}
\usepackage[all]{xy}
\usepackage[dvips]{graphicx}

\newtheorem{thm}{Theorem}[section]
\newtheorem{coro}[thm]{Corollary}
\newtheorem{lemma}[thm]{Lemma}
\newtheorem{prop}[thm]{Proposition}
\newtheorem{conj}[thm]{Conjecture}
\newtheorem{Ques}[thm]{Question}

\theoremstyle{remark}
\newtheorem{remark}[thm]{\textbf{Remark}}

\theoremstyle{definition}

\numberwithin{equation}{thm}
\newcommand{\newpara}{\noindent\refstepcounter{thm}{\bf(\thethm)\;}}  

\newcommand{\Br}{\mathrm{Br}}
\newcommand{\rM}{\mathrm{M}}
\newcommand{\Int}{\mathrm{Int}}
\newcommand{\car}{\mathrm{char}}
\newcommand{\Nrd}{\mathrm{Nrd}}
\newcommand{\ind}{\mathrm{ind}}

\newcommand{\sR}{\mathscr{R}}

\newcommand{\bfSL}{\mathbf{SL}}
\newcommand{\bfSU}{\mathbf{SU}}
\newcommand{\bfSO}{\mathbf{SO}}
\newcommand{\bfU}{\mathbf{U}}
\newcommand{\bfSpin}{\mathbf{Spin}}

\newcommand{\GL}{\mathrm GL}
\newcommand{\Sn}{\mathrm{Sn}}
\newcommand{\rank}{\mathrm{rank}}
\newcommand{\disc}{\mathrm{disc}}
\newcommand{\fH}{\mathfrak{H}}
\newcommand{\Gal}{\mathrm{Gal}}

\newcommand{\End}{\mathrm{End}}

\newcommand{\oi}{\mathrm{i}}
\newcommand{\ii}{\mathrm{ii}}

\newcommand{\matr}[4]{\begin{pmatrix} {#1}& {#2}\\ {#3}&{#4}\end{pmatrix}}
\newcommand{\smatr}[4]{\bigl(\begin{smallmatrix} {#1}& {#2}\\ {#3}&{#4}\end{smallmatrix}\bigl)}

\newcommand{\al}{\alpha}
\newcommand{\veps}{\varepsilon}

\newcommand{\simto}{\xrightarrow{\sim}}

\newcommand{\lra}{\longrightarrow}
\newcommand{\set}[1]{\{\,{#1}\,\}}
\newcommand{\wt}[1]{\widetilde{#1}}

\begin{document}
\title{\textbf{Hasse Principle for Simply Connected Groups over Function Fields of Surfaces}}
\author{Yong HU}
\date{}

\maketitle


\


\begin{abstract}
Let $K$ be the function field of a $p$-adic curve, $G$ a semisimple simply connected group over $K$ and $X$
a $G$-torsor over $K$. A conjecture of Colliot-Th\'el\`ene, Parimala and Suresh predicts that if for every discrete valuation
$v$ of $K$, $X$ has a point over the completion $K_v$, then $X$ has a $K$-rational point. The main result  of this paper is the proof of this
conjecture for groups of some classical types. In particular, we prove the conjecture when $G$ is of one of the following types: (1) ${}^2A_n^*$, i.e.
$G=\mathbf{SU}(h)$ is the special unitary group of  some hermitian form $h$ over a pair $(D\,,\,\tau)$, where $D$ is a central division algebra of square-free index over a quadratic extension $L$ of $K$ and $\tau$ is an involution of the second kind on $D$ such that $L^{\tau}=K$; (2) $B_n$, i.e.,
$G=\mathbf{Spin}(q)$ is the spinor group of quadratic form of odd dimension over $K$; (3) $D_n^*$, i.e., $G=\mathbf{Spin}(h)$
is the spinor group of a hermitian form $h$ over a quaternion $K$-algebra $D$ with an orthogonal involution. Our method actually yields a parallel local-global
result over the fraction field of a $2$-dimensional, henselian, excellent local domain with finite residue field, under suitable assumption on
the residue characteristic.
\end{abstract}

\ \  {\bf MSC classes:} \ 11E72, 11E57

\section{Introduction}

Let $K$ be a field and $G$  a smooth connected linear algebraic
group over $K$. The cohomology set $H^1(K,\,G)$ classifies up to
isomorphism $G$-torsors over $K$, and a class $\xi\in H^1(K,\,G)$ is
trivial if and only if the corresponding $G$-torsor has a
$K$-rational point. Let $\Omega_K$ denote the set of (normalized)
discrete valuations (of rank 1) of the field $K$. For each
$v\in\Omega_K$, let $K_v$ denote the completion of $K$ at $v$. The
restriction maps $H^1(K,\,G)\to H^1(K_v,\,G)$, $v\in\Omega_K$ induce
a natural map of pointed sets
\[
H^1(K,\,G)\lra\prod_{v\in \Omega_K}H^1(K_v,\,G)\,.
\]If the kernel of this map is trivial, we say that the \emph{Hasse
principle} with respect to $\Omega_K$ holds for $G$-torsors over
$K$.

In the case of a $p$-adic function field, by which we mean the
function field of an algebraic curve over a $p$-adic field (i.e., a
finite extension of $\mathbb{Q}_p$), the following conjecture was made by
Colliot-Th\'el\`ene, Parimala and Suresh.

\begin{conj}[{\cite{CTPaSu}}]\label{conj1p1NEW}
Let $K$ be the function field of an algebraic curve over a $p$-adic
field and let $G$ be a semisimple simply connected group over $K$.

Then  the kernel of the natural map
\[
H^1(K,\,G)\lra\prod_{v\in\Omega_K}H^1(K_v,\,G)
\]is trivial. In other words, if a $G$-torsor has points in all
completions $K_v\,,\,v\in\Omega_K$, then it has a $K$-rational
point.
\end{conj}

\newpara\label{para1p2NEWTEMP}
Let $K$ be a $p$-adic function field with field of  constants  $F$,
i.e.,  $K$ is the function field of a smooth projective
geometrically integral curve over the $p$-adic field $F$. Let $A$ be
the ring of integers of $F$. It is in particular a henselian
excellent local domain of dimension 1.  By resolution of
singularities, there exists a proper flat morphism $\mathcal{X}\to
\mathrm{Spec} A$, where $\mathcal{X}$ is a connected regular 2-dimensional
scheme with function field $K$. We will say that $\mathcal{X}$ is a
$p$-adic \emph{arithmetic surface} with function field $K$, or that
$\mathcal{X}\to\mathrm{Spec} A$ is a \emph{regular proper model} of the
$p$-adic function field $K$.

An analog in the context of a 2-dimensional base is as follows. Let
$A$ be a henselian excellent 2-dimensional local domain with
\emph{finite} residue field $k$ and let $K$ be the field of
fractions of $A$. Again by resolution of singularities, there exists
a proper birational morphism $\mathcal{X}\to \mathrm{Spec} A$, where
$\mathcal{X}$ is a connected regular 2-dimensional scheme with
function field $K$. We will say that $\mathrm{Spec} A$ is a \emph{local
henselian surface} with function field $K$ and that
$\mathcal{X}\to\mathrm{Spec} A$ is a \emph{regular proper model} of $\mathrm{Spec}
A$.

Experts have also been interested in the following analog of
Conjecture$\;$\ref{conj1p1NEW}:

\begin{Ques}\label{ques1p2NEW}
Let $K$ be the function field of a local henselian surface $\mathrm{Spec} A$
with finite residue field and let $G$ be a semisimple simply
connected group over $K$.

Does the Hasse principle with respect to $\Omega_K$ hold for
$G$-torsors over $K\,?$
\end{Ques}

Let $K$ be the function field of a $p$-adic arithmetic surface or a
local henselian surface with finite residue field. For most
quasi-split $K$-groups, the Hasse principle may be proved by
combining an injectivity property of the Rost invariant map (cf.
\cite[Thm.$\;$5.3]{CTPaSu}) and results from higher dimensional
class field theory of Kato and Saito.

The goal of this paper is to prove the Hasse principle for groups of
several  types in the non-quasisplit case. To give precise statement
of our main result, we will refine the usual classification of
absolutely simple simply connected groups in some cases.

\

\newpara\label{para1p4NEWTEMP}
Let $E$ be a field and let $G$ be an absolutely simple simply
connected group over $E$. We say that  $G$ is of type

(1) ${}^1A_n^*$, if $G=\bfSL_1(A)$ is the special linear group of
some central simple $E$-algebra $A$ \emph{of square-free index};

(2) ${}^2A_n^*$, if $G=\bfSU(h)$ is the special unitary group of
some nonsingular hermitian form $h$ over a pair $(D,\,\tau)$, where
$D$ is a central division algebra \emph{of square-free index}
over a separable quadratic field extension $L$ of $E$ and $\tau$ is
an involution of the second kind on $D$ such that $L^{\tau}=E$; when the index of division algebra $D$
is odd (resp. even), we say the group $G=\bfSU(h)$ is of type ${}^2A_n^*$ of odd (resp. even) index;

(3) $C_n^*$, if $G=\bfU(h)$ is the unitary group (also called
symplectic group) of a nonsingular hermitian form $h$ over a pair
$(D,\,\tau)$, where $D$ is \emph{quaternion algebra} over $E$ and
$\tau$ is a symplectic involution on $D$;

(4) $D_n^*$ (in characteristic $\neq 2$), if $G=\bfSpin(h)$ is the
spin group of a nonsingular hermitian form $h$ over a pair
$(D,\,\tau)$, where $D$ is \emph{quaternion algebra} over $E$ and
$\tau$ is an orthogonal involution on $D$;

(5) $F_4^{red}$ (in characteristic different from $2,\,3$), if
$G=\mathbf{Aut}_{alg}(J)$ is the group of algebra automorphisms of
some \emph{reduced} exceptional Jordan $E$-algebra $J$ of dimension
27.

Recall also that $G$ is of type

(6) $B_n$ (in characteristic $\neq 2$), if $G=\bfSpin(q)$ is the
spin group of a nonsingular quadratic form $q$ of dimension $2n+1$
over $E$;

(7) $G_2$ (in characteristic $\neq 2$), if $G=\mathbf{Aut}_{alg}(C)$
is the group of algebra automorphisms of a Cayley algebra $C$ over
$E$.

\

\newpara\label{para1p5NEWTEMP}
In the local henselian case, we shall exclude some possibilities for
the residue characteristic. To this end, we define for any
semisimple simply connected group $G$ a set $S(G)$ of prime numbers
as follows (cf. \cite[$\S$2.2]{Ser94} or \cite[p.44]{Gille10}):

$S(G)=\set{2}$, if $G$ is of type $G_2$ or of classical
type $B_n,\,C_n$ or $D_n$ (trialitarian $D_4$ excluded);

$S(G)=\set{2\,,\,3}$, if $G$ is of type $E_6,\,E_7,\,F_4$ or
trialitarian $D_4$;

$S(G)=\set{2\,,\,3\,,\,5}$, if $G$ is of type $E_8$;

$S(G)$ is the set of prime factors of the index $\ind(A)$ of $A$, if
$G=\bfSL_1(A)$ for some central simple algebra $A$;

$S(G)$ is the set of prime factors of $2.\ind(D)$, if $G=\bfSU(h)$
for some nonsingular hermitian form $h$ over a division algebra $D$
with an involution of the second kind.

In the general case, define $S(G)=\cup S(G_i)$, where $G_i$ runs
over the almost simple factors of $G$.

When $G$ is absolutely simple, let $n_G$ be the order of the Rost invariant of $G$. Except for a few cases where
$n_G=1$, the set $S(G)$ coincides with the set of prime factors of $n_G$ (cf.
\cite[Appendix$\;$B]{GMS} or \cite[$\S$31.B]{KMRT}).

\

We summarize our  main results in the following two theorems.

\begin{thm}\label{thm1p3NEW}\footnote{
R. Preeti \cite{Preeti} has proved results on the injectivity of the Rost invariant which overlap with
the results in this paper. Our work was carried out independently.}
Let $K$ be the function field of a $p$-adic arithmetic surface and $G$ a semisimple simply connected group over $K$. Assume $p\neq 2$ if $G$ contains an almost simple factor of type ${}^2A_n^*$ of even index.

If every almost simple factor of $G$ is of type
\[
{}^1A_n^*\,,\,{}^2A^*_n\,,\,B_n\,,\,C_n^*\,,\,D_n^*\,,\,F_4^{red}\,\,\text{
or }\;\;G_2\,,
\]then the natural map
\[
H^1(K,\,G)\lra \prod_{v\in\Omega_K}H^1(K_v\,,\,G)
\]has a trivial kernel.
\end{thm}

\begin{thm}\label{thm1p4NEWv2}
Let $K$ be the function field of a local henselien surface with finite residue field of characteristic
$p$. Let $G$ be a semisimple simply connected group over $K$. Assume
$p\notin S(G)$.

If every almost simple factor of $G$ is of type
\[
{}^1A_n^*\,,\,{}^2A^*_n \text{ of odd index}\,,\,\,B_n\,,\,C_n^*\,,\,D_n^*\,,\,F_4^{red}\,\,\text{
or }\;\;G_2\,,
\]then the natural map
\[
H^1(K,\,G)\lra \prod_{v\in\Omega_K}H^1(K_v\,,\,G)
\]has a trivial kernel.

If moreover  the Hasse principle with respect to $\Omega_K$ holds for quadratic forms $q$ of rank $6$ over $K ($i.e., $q$ has a nontrivial zero over $K$
if and only if it has a nontrivial zero over every $K_v\,,\,v\in\Omega_K)$, then the same result is also true for an absolutely simple group
of type ${}^2A_n^*$ of even index.
\end{thm}

In fact, it suffices to consider only divisorial discrete valuations
in the above theorems.

\begin{remark}\label{remark1p8v2}
Let $K$ be as in Theorem$\;$\ref{thm1p3NEW} or \ref{thm1p4NEWv2}. Assume the residue characteristic $p$ is not 2.

(1) By \cite[Thm.$\;$3.4]{Salt97} (the arithmetic case) and \cite[Thm.$\;$3.4]{Hu11} (the local henselian case),
a central division algebra of exponent 2 over the field $K$ is either a quaternion algebra or a biquaternion algebra. So
for a group of type $C_n$, say $G=\bfU(h)$ with $h$ a hermitian form over a symplectic pair $(D,\,\tau)$,  the only case not covered by our theorems is the case where $D$  is a biquaternion algebra. Similarly, for a group of classical type $D_n$, say $D=\bfSpin(h)$ with $h$ a hermitian form over an orthogonal pair
$(D,\,\tau)$, the only remaining case is the one with $D$ a biquaternion algebra.

(2) In Theorem$\;$\ref{thm1p4NEWv2}, the hypothesis on the Hasse principle for quadratic forms of rank 6 is satisfied if
$K=\mathrm{Frac}(\mathcal{O}[\![t]\!])$ is the fraction field of a formal power series ring over a complete discrete valuation ring
$\mathcal{O}$ (whose residue field is finite), by \cite[Thm.$\;$1.2]{Hu10}. In the arithmetic case, this is established in \cite[Thm.$\;$3.1]{CTPaSu}.
\end{remark}

In the rest of the paper, after some preliminary reviews in  Section$\;$\ref{sec2},  we will prove our main theorems case by case: the cases ${}^1A_n^*$, $C_n$, $F_4^{red}$ and $G_2$ in Section$\;$\ref{sec3}; the cases $B_n$ and $D_n^*$ in Sections$\;$\ref{sec4} and \ref{sec5}; and the case ${}^2A_n^*$ in Section$\;$\ref{sec6}.

Our proofs use ideas from Parimala and Preeti's paper \cite{PaPr}. In particular, two exact sequences of Witt groups, due to
Parimala--Sridharan--Suresh and Suresh respectively, play a special role in some cases. Other important
ingredients include Hasse principles for degree 3 cohomology of
$\mathbb{Q}/\mathbb{Z}(2)$ coming from higher dimensional class field theory of Kato
and Saito (cf. \cite{Ka86} and \cite{Sai87}), as well as the work of
Merkurjev and Suslin on reduced norm criterion and norm principles
(\cite{Suslin85}, \cite{Mer96}). For spinor groups and groups of type ${}^2A_n^*$ of even index,  we also make use of
results on quadratic forms over the base field $K$ obtained in
\cite{PaSu}, \cite{Le10} (see also \cite{HHK}) in the $p$-adic case
and in \cite{Hu11} in the local henselian case.

\section{Some reviews and basic tools}\label{sec2}

In this section, we briefly review some basic notions which will be used
frequently and we recall some known results that are essential in the
proofs to come later.

Throughout this section, let $L$ denote a field of characteristic
different from 2.

\subsection{Hermitian forms and Witt groups}

We will assume the readers have basic familiarity with the theory of
involutions and hermitian forms over central simple algebras (cf.
\cite{Schar}, \cite{Knus}, \cite{KMRT}). For later use, we recall in
this subsection some facts on Witt groups, the ``key exact
sequence'' of Parimala, Sridharan and Suresh  and the exact sequence of Suresh. The readers are referred to \cite[$\S$3 and
Appendix$\;$2]{BP1}, \cite[$\S$3]{BP2} and \cite[$\S$8]{PaPr} for more information.

Unless otherwise stated, all hermitian forms and skew-hermitian forms
(in particular all quadratic forms) in this paper are assumed to be
nonsingular.

\

\newpara\label{para2p1NEW} Let $L$ be a field of characteristic
different from 2, $A$ a central simple algebra over $L$ and $\sigma$
an involution on $A$. Let $E=L^{\sigma}$. We say that $\sigma$ is an
$L/E$-\emph{involution} on $A$. To each hermitian or skew-hermitian
form $(V,\,h)$ over $(A,\,\sigma)$, one can associate an involution
on $\End_A(V)$, called the \emph{adjoint involution} on $\End_A(V)$
with respect to $h$. This is the unique involution $\sigma_h$ on
$\End_A(V)$ such that
\[
h(x,\,f(y))=h(\sigma_h(f)(x)\,,\,y)\,,\quad\forall\;x,\,y\in
V\,,\;\;\forall\;f\in\End_A(V)\,.
\]

For a fixed finitely generated right $A$-module $V$, define an
equivalence relation $\sim$ on the set of hermitian or
skew-hermitian forms on $V$ (with respect to the involution
$\sigma$)  by
\[
h\,\sim\, h'\;\iff\; \text{there exists }\;\lambda\in E^* \text{
such that }h=\lambda.h'\,.
\]Let $\mathcal{H}^+(V)$ (resp. $\mathcal{H}^-(V)$) denote the set of
equivalence classes of hermitian (resp. skew-hermitian) forms on $V$
and let $\mathcal{H}^{\pm}(V)=\mathcal{H}^+(V)\cup\mathcal{H}^-(V)$.
The assignment $h\mapsto \sigma_h$ defines a map from
$\mathcal{H}^{\pm }(V)$ to the set of involutions on $\End_A(V)$. If
$\sigma$ is of the first kind, then the map $h\mapsto \sigma_h$
induces a bijection between $\mathcal{H}^{\pm}(V)$ and the set of
involutions of the first kind on $\End_A(V)$, and the involutions
$\sigma_h$ and $\sigma$ have the same type (orthogonal or
symplectic) if $h$ is hermitian and they have opposite types if $h$
is skew-hermitian. If $\sigma$ is of the second kind, then the map
$h\mapsto \sigma_h$ induces a bijection between $\mathcal{H}^+(V)$
and the set of $L/E$-involutions on $\End_A(V)$. (cf. \cite[p.43,
Thm.$\;$4.2]{KMRT}.)

If $A=L$ and $\sigma=\mathrm{id}$, a hermitian (resp. skew-hermitian) form
$h$ is simply a symmetric (resp. skew-symmetric) bilinear form $b$.
In this case, $b\mapsto \sigma_b$ defines a bijection between
equivalence classes of nonsingular symmetric or skew-symmetric
bilinear forms on $V$ modulo multiplication by a factor in $L^*$ and
involutions of the first kind on $\End_L(V)$. If $q$ is the
quadratic form associated to a symmetric bilinear form $b$, we also
write $\sigma_q$ for the adjoint involution $\sigma_b$.

\

\newpara\label{para2p2NEW} Let $(A,\,\sigma)$ be a pair consisting of a central simple algebra $A$ over a field $L$
of characteristic $\neq 2$ and an involution (of any kind) $\sigma$
on $A$. The orthogonal sum of hermitian forms defines a semigroup
structure on the set of isomorphism classes of hermitian forms over
$(A,\,\sigma)$. The quotient of the corresponding Grothendieck group
by the subgroup generated by hyperbolic forms is called the
\emph{Witt group} of $(A,\,\sigma)$ and denoted
$W(A,\,\sigma)=W^1(A,\,\sigma)$. The same construction applies to
skew-hermitian forms and the corresponding Witt group will be
denoted $W^{-1}(A,\,\sigma)$.

If $A=L$ and $\sigma=\mathrm{id}$, then $W(A,\,\sigma)$ is the usual Witt
group $W(L)$ of quadratic forms (cf. \cite{Lam}, \cite{Schar}). One
has a ring structure on $W(L)$ induced by the tensor product of
quadratic forms. The classes of  even dimensional forms form an
ideal $I(L)$ of the ring $W(L)$. For each $n\ge 1$, we write
$I^n(L)$ for the $n$-th power of the ideal $I(L)$. As an abelian
group, $I^n(L)$ is generated by the classes of $n$-fold Pfister
forms.

\

\newpara\label{para2p3NEW} Let $D$ be a quaternion division algebra
over a field $L$ of characteristic $\neq 2$. Let $\tau_0$ be the
standard (symplectic) involution on $D$. The Witt group
$W(D,\,\tau_0)$ has a nice description as follows (cf.
\cite[p.352]{Schar}).

If $h: V\times V\to D$ is a hermitian form over $(D,\,\tau_0)$, then
the map
\[
q_h\,:V\lra L\,,\quad q_h(x):=h(x,\,x)
\]defines a quadratic form on the $L$-vector space $V$, called the
\emph{trace form} of $h$. If $h$ is isomorphic to the diagonal form
$\langle \lambda_1\,,\dotsc, \lambda_r\rangle$, then $q_{h}$ is
isomorphic to the form $\langle \lambda_1\,,\dotsc,
\lambda_r\rangle\otimes n_D$, where $n_D$ denotes the norm form of
the quaternion algebra $D$. By \cite[p.352, Thm.$\;$10.1.7]{Schar},
the assignment $h\mapsto q_h$ induces an injective group
homomorphism $W(D,\,\tau_0)\to W(L)$, whose image is the principal
ideal of $W(L)$ generated by (the class of) the norm form $n_D$ of
$D$. In particular, two hermitian forms over $(D,\,\tau_0)$ are
isomorphic if and only if their trace forms are isomorphic.

\

\newpara\label{para2p4NEW} Let $L/E$ be a quadratic extension of
fields of characteristic different from 2. The nontrivial element
$\iota$ of the Galois group $\Gal(L/E)$ may be viewed as a unitary
involution on the $L$-algebra $A=L$. The Witt group $W(L,\,\iota)$
can be determined as follows (cf. \cite[pp.348--349]{Schar}):

As in (\ref{para2p3NEW}), to each hermitian form $h: V\times V\to L$
over $(L,\,\iota)$, one can associate a quadratic form  $q_h$ on the
$E$-vector space $V$, called the \emph{trace form} of $h$, by
defining
\[
q_h(x):=h(x,\,x)\in E\,,\,\forall\;x\in V\,.
\]One can show that $h\mapsto q_h$ induces a group homomorphism $W(L,\,\iota)\to W(E)$
which identifies $W(L,\,\iota)$ with the kernel of the base change
homomorphism $W(E)\to W(L)$. In particular, two hermitian forms over
$(L,\,\iota)$ are isomorphic if and only if their trace forms are
isomorphic. (cf. \cite[Thm.$\;$10.1.2]{Schar}.)

Let $\delta\in E$ be an element such that $L=E(\sqrt{\delta})$. Then
for  $a\in E^*$, the trace form of $h=\langle a\rangle$ is
isomorphic to $\langle a\,,\,-a\delta\rangle=a.\langle
1\,,\,-\delta\rangle$. So the image of the map
\[
W(L\,,\,\iota)\lra W(E)\,;\quad h\mapsto q_h
\]is the principal ideal generated by the form $\langle 1\,,\,-\delta\rangle$ (cf.
\cite[Remark$\;$10.1.3]{Schar}).

\

\newpara\label{para2p5NEW} Let $A$ be a central simple algebra over
a field $L$ of characteristic $\car(L)\neq 2$. Let $\sigma$ be an
involution on $A$ and let $E=L^{\sigma}$. For any invertible element
$u\in A^*$, let $\Int(u): A\to A$ denote the inner automorphism
$x\mapsto u.x.u^{-1}$. If $\sigma(u)u^{-1}=\pm 1$, then
$\Int(u)\circ \sigma$ is an involution on $A$ of the same kind as
$\sigma$.

Conversely, let $\sigma,\,\tau$ be involutions of the same kind on
$A$. If $\sigma$ and $\tau$ are of the first kind, then there is a
unit $u\in A^*$, uniquely determined up to a scalar factor in $E^*$,
such that $\tau=\Int(u)\circ \sigma$ and $\sigma(u)=\pm u$.
Moreover, the two involutions $\sigma$ and $\tau=\Int(u)\circ
\sigma$ are of the same type (orthogonal or symplectic) if and only
if $\sigma(u)=u$. If $\sigma$ and $\tau$ are of the second kind,
then there exists a unit $u\in A^*$, uniquely determined up to a
scalar factor in $E^*$, such that $\tau=\Int(u)\circ \sigma$ and
$\sigma(u)=u$.

Let $\fH(A,\,\sigma)=\fH^1(A,\,\sigma)$ (resp.
$\fH^{-1}(A,\,\sigma)$) denote the category of hermitian (resp.
skew-hermitian) forms over $(A,\,\sigma)$. Let
$\veps,\,\veps'\in\{\pm 1\}$. Let $a\in A^*$ be an element such that
$\sigma(a)=\veps' a$. Then the functor
\[
\Phi_a\,:\;\;\fH^{\veps}(A\,,\,\Int(a^{-1})\circ \sigma)\lra
\fH^{\veps\veps'}(A,\,\sigma)\,;\quad (V,\,h)\longmapsto (V,\,a.h)
\]is an equivalence of categories, called a \emph{scaling}. There is
also an induced isomorphism of Witt groups
\[
\phi_a\,:\;\;W^{\veps}(A\,,\,\Int(a^{-1})\circ\sigma)\simto
W^{\veps\veps'}(A,\,\sigma)\,.
\]
In particular, if $\sigma$ and $\tau$ are involutions of the same
kind and type on $A$, then there is a scaling isomorphism of Witt
groups $\phi_a: W(A,\,\tau)\simto W(A,\,\sigma)$.

\

\newpara\label{para2p6NEW} Let $A$ be a central simple algebra over
a field $L$ of characteristic $\neq 2$ and $\sigma$ an involution of
any kind on $A$. Let $(V,\,h)$ be a hermitian form over
$(A,\,\sigma)$. Let $B=\End_A(V)$ and let $\sigma_h$ be the adjoint
involution with respect to $h$. There is an equivalence of
categories, called the \emph{Morita equivalence},
\[
\Phi_h\,:\;\;\fH(B\,,\,\sigma_h)\lra \fH(A\,,\,\sigma)
\]defined as follows (cf. \cite[$\S$1.4]{BP1},
\cite[$\S$I.9]{Knus}): For a hermitian form $(M,\,f)$ over
$(B,\sigma_h)$, define a map
\[
h*f\,:\;\;(M\otimes_BV)\times (M\otimes_BV)\lra A
\]by
\[
(h*f)(m_1\otimes v_1\,,\,m_2\otimes
v_2):=h(v_1\,,\,f(m_1,\,m_2)(v_2))\,.
\]One verifies that $\Phi_h(M\,,\,f):=(M\otimes_BV\,,\,h*f)$ yields
a well-defined functor $\fH(B\,,\,\sigma_h)\to \fH(A\,,\,\sigma)$,
which can be shown to be an equivalence (cf. \cite[p.56,
Thm.$\;$I.9.3.5]{Knus}). The Morita equivalence induces an
isomorphism of Witt groups:
\[
\phi_h\,:\;\;W(\End_A(V)\,,\,\sigma_h)\simto W(A\,,\,\sigma)\,.
\]

\newpara\label{para2p7NEW} We briefly recall the construction of the
key exact sequence of Parimala, Sridharan and Suresh. The readers
are referred to \cite[$\S$3 and Appendix$\;$2]{BP1} for more
details.

Let $(A,\,\sigma)$ be a central simple algebra with involution over
$L$. Let $E=L^{\sigma}$. Assume there is a subfield $M\subseteq A$
which is a quadratic extension of $L$ such that $\sigma(M)=M$.
Suppose $\sigma|_M=\mathrm{id}_M$ if $\sigma$ is of the first kind. Let
\[
\wt{A}:=\{a\in A\,|\,a.m=m.a\,,\forall\; m\in M\}
\]be the centralizer of $M$ in $A$. This is a central simple algebra over $M$. By \cite[Lemma$\;$3.1.1]{BP1},
there exists $\mu\in A^*$ such that $\sigma(\mu)=-\mu$ and that the
restriction of $\Int(\mu)$ to $M$ is the nontrivial element of the
Galois group $\Gal(M/L)$.

Set $\tau=\Int(\mu)\circ\sigma$ and let $\tau_1,\,\tau_2$ be the
restrictions of $\tau$ and $\sigma$ to $\wt{A}$ respectively. Then
$\tau_1$ is an involution of the second kind, $\tau_2$ is of the
same kind and type as $\sigma$, and $\tau$ is orthogonal (resp.
symplectic) if and only if $\sigma$ is symplectic (resp. orthogonal).

One has a decomposition $A=\wt{A}\oplus\mu.\wt{A}$ (as right
$M$-modules). Let $\pi_1,\,\pi_2: A\to\wt{A}$ be the $M$-linear
projections
\[
\pi_1(x+\mu y)=x\,,\quad \pi_2(x+\mu
y)=y\,,\quad\forall\;x,\,y\in\wt{A}\,.
\]These induce well-defined group homomorphisms
\[
\pi_1\,:\;W(A,\,\tau)\lra W(\wt{A},\,\tau_1) \quad\text{and }\quad
\pi_2\,:\;W^{-1}(A,\,\tau)\lra W(\wt{A}\,,\,\tau_2)\,.
\]On the other hand, let $\lambda\in M$ be an element such that $\lambda^2\in
L$ and $M=L(\lambda)$. For a hermitian form $(\wt{V},\,f)$ over
$(\wt{A},\,\tau_1)$, define $\rho(f)$ to be the unique
skew-hermitian form on $V=\wt{V}\oplus \wt{V}\mu$ which extends
$\lambda.f: \wt{V}\times\wt{V}\to \wt{A}$. This defines a group
homomorphism
\[
\rho\,:\;W(\wt{A},\,\tau_1)\lra W^{-1}(A,\,\tau)\,;\quad
(\wt{V},\,f)\mapsto (\wt{V}\oplus\wt{V}\mu\,,\,\rho(f))\,.
\]

The sequence
\begin{equation}\label{eq2p7p1NEW}
W^{\veps}(A\,,\,\tau)\overset{\pi_1}{\lra}
W^{\veps}(\wt{A}\,,\,\tau_1)\overset{\rho}{\lra}W^{-\veps}(A\,,\,\tau)\overset{\pi_2}{\lra}
W^{\veps}(\wt{A}\,,\,\tau_2)
\end{equation}turns out to be an exact
sequence (cf. \cite[Appendix$\;$2]{BP1}).

Since $\tau(\mu)=-\mu$, one has a scaling isomorphism (cf.
(\ref{para2p5NEW}))
\[
\phi_{\mu}^{-1}\,:\;W^{-1}(A\,,\,\tau)\overset{\sim}{\lra}
W(A\,,\,\sigma).
\]We may thus replace $W^{-1}(A\,,\,\tau)$ in the exact sequence \eqref{eq2p7p1NEW} by $W(A,\,\sigma)$
 and rewrite it as
\begin{equation}\label{eq2p7p2NEW}
W(A\,,\,\tau)\overset{\pi_1}{\lra}W(\wt{A}\,,\,\tau_1)\overset{\wt{\rho}}{\lra}W(A\,,\,\sigma)\overset{\wt{\pi}_2}{\lra}W(\wt{A}\,,\,\tau_2)\,
\end{equation}where $\wt{\rho}=\phi_{\mu}^{-1}\circ\rho$ and $\wt{\pi}_2=\pi_2\circ\phi_{\mu}$. This exact sequence is due to Parimala, Sridharan and Suresh and is referred to as the \emph{key exact sequence} in \cite{BP1}.

We will only use the exact sequence \eqref{eq2p7p2NEW} in the case
where $A=D$ is a quaternion algebra and $\sigma$ is an orthogonal
involution. This special case was already discussed by Scharlau in
\cite[p.359]{Schar}.

\

\newpara\label{para2p8v2} Now let $D$ be a quaternion division algebra over a quadratic field extension $L$ of $E$ and let $\tau$ be a unitary
 $L/E$-involution on $D$ (i.e. a unitary involution such that $L^{\tau}=E$).
There is a unique quaternion $E$-algebra $D_0$ contained in $D$ such that
$D=D_0\otimes_EL$ and $\tau=\tau_0\otimes\iota$, where $\tau_0$ is the canonical (symplectic) involution on
$D_0$ and $\iota$ is the nontrivial element of the Galois group $\Gal(L/E)$.  Write $L=E(\sqrt{d})$ with $d\in E^*$. Then $D=D_0\oplus D_0\sqrt{d}$. For any hermitian
form $(V,\,h)$ over $(D,\,\tau)$, we may write
\[
h(x,\,y)=h_1(x,\,y)+h_2(x,\,y)\sqrt{d}\quad \text{ with }\;h_i(x,\,y)\in D_0\,,\quad\text{for }\; i=1,\,2
\]for any $x,\,y\in V$.

The projection $h\mapsto h_2$ defines a group homomorphism
\[
 p_2\,:\;\;W(D,\,\tau)\lra W^{-1}(D_0\,,\,\tau_0)\,.\]
For  a hermitian form $(V_0,\,f)$   over $(D_0,\,\tau_0)$, set
\[
V=V_0\otimes_{D_0}D=V_0\otimes_EL=V_0\oplus V_0\sqrt{d}\,
\]and let $\wt{\rho}(f)\,:\;V\times V\to D$ be the map extending $f:V_0\times V_0\to D_0$ by $\tau$-sesquilinearity. One checks that this defines a group homomorphism
\[
\wt{\rho}\,:\;\;W(D_0,\,\tau_0)\lra W(D,\,\tau)\,;\quad (V_0,\,f)\longmapsto (V_0\oplus V_0\sqrt{d}\,,\,\wt{\rho}(f))\,.
\]
For any quadratic form $q$ over $L=E(\sqrt{d})$, there are quadratic forms $q_1,\,q_2$ over $k$ such that $q(x)=q_1(x)+q_2(x)\sqrt{d}$. We have
thus group homomorphisms
\[
\pi_i\,:\;\;W(L)\lra W(E)\,;\;\;q\longmapsto q_i\,,\quad i=1,\,2\,.
\]
We denote by $\wt{\pi}_1\,:\;W(L)\to W(D_0,\,\tau_0)$ the composite map
\[
W(L)\overset{\pi_1}{\lra} W(E)\lra W(D_0,\,\tau_0)
\]where the map $W(E)\to W(D_0,\,\tau_0)$ is induced by base change.

Suresh (cf. \cite[Prop.$\;$8.1]{PaPr})
 proved that the  sequence
\[
W(L)\overset{\wt{\pi}_1}{\lra}W(D_0\,,\,\tau_0)\overset{\wt{\rho}}{\lra}W(D\,,\,\tau)\overset{p_2}{\lra}W^{-1}(D_0,\,\tau_0)
\]is exact. We will refer to this sequence as Suresh's exact sequence in the sequel.

\subsection{Invariants of hermitian forms}\label{sec2p2}

In this subsection, we recall the definitions of some invariants of
hermitian forms. For more details, see \cite[$\S$2]{BP1},
\cite[$\S$3]{BP2} and \cite[$\S$5, $\S$7]{PaPr}.

\

\newpara\label{para2p8NEW} Let $(D,\,\sigma)$ be a central division algebra with involution over
$L$. Let $E=L^{\sigma}$. Let $(V,\,h)$ be a  hermitian form over
$(D,\,\sigma)$. The \emph{rank} of $(V,\,h)$, denoted $\rank(V,\,h)$
or simply $\rank(h)$, is by definition the rank of the $D$-module
$V$:
\[
\rank(h):=\rank_D(V)\,.
\]

\

\newpara\label{para2p9NEW} With notation as in (\ref{para2p8NEW}), let $e_1,\dotsc, e_n$ be a basis of the $D$-module $V$ (so that
$\rank(h)=\rank_D(V)=n$). Let $M(h):=(h(e_i,\,e_j))$ be the matrix
of the hermitian form $h$ with respect to this basis. The matrix
algebra $A=\rM_n(D)$ has dimension
\[
\dim_LA=n^2\dim_LD=(\rank(h).\,\deg_LD)^2\,.
\]Put
\[
m=\sqrt{\dim_LA}=\rank(h).\,\deg_LD=\frac{\dim_LV}{\deg_LD}\,.
\]We define the \emph{discriminant} $\disc(h)=\disc(V,\,h)$ of the  hermitian form $(V,\,h)$ by
\[
\disc(h)=(-1)^{\frac{m(m-1)}{2}}\Nrd_A(M(h))\;\in\;\begin{cases}
E^*/E^{*2}\;\;&\; \text{ if $\sigma$ is of the first kind}\; \\
E^*/N_{L/E}(E^*)\;\;& \; \text{ if $\sigma$ is of the second kind}
\end{cases}
\]If $h$ is a hermitian form over $(D,\,\sigma)$, the image of the
canonical map
\[
H^1(E\,,\,\bfSU(h))\lra H^1(E\,,\,\bfU(h))
\]consists of classes $[h']\in H^1(E\,,\,\bfU(h))$ of hermitian
forms $h'$ which have the same rank and discriminant as $h$.

\

\newpara\label{para2p10NEW} Let $D$ be a cenral division algebra over
$L$ and let $\sigma$ be an orthogonal involution on $D$. Note that
the Brauer class of $D$ in the Brauer group $\Br(L)$ lies in the
subgroup
\[
{}_2\Br(L):=\{\al\in\Br(L)\,|\,2.\al=0\}\,.
\]

Let $h$ be a hermitian form over $(D,\,\sigma)$. Let
\[
\delta\,:\;\;H^1(L\,,\,\bfSU(h))\lra H^2(L\,,\,\mu_2)={}_2\Br(L)
\]be the connecting map associated to the exact sequence of algebraic groups
\[
1\lra \mu_2\lra\bfSpin(h)\lra\bfSU(h)\lra 1\,.
\]Let $h'$ be a hermitian form over $(D,\,\sigma)$ such that
$\rank(h')=\rank(h)$ and $\disc(h')=\disc(h)$. Then there is an
element $c(h')\in H^1(L,\,\bfSU(h))$ which lifts $[h']\in
H^1(L\,,\,\bfU(h))$. The class of $\delta(c(h'))$ in the quotient
${}_2\Br(L)/\langle [D]\rangle$ is independent of the choice of
$c(h')$ (cf. \cite[$\S$2.1]{BP1}).  Following \cite{Bar}, we define
the \emph{relative Clifford invariant} $\mathscr{C}\ell_h(h')$ by
\[
\mathscr{C}\ell_h(h'):=[\delta(c(h'))]\;\in\;\frac{{}_2\Br(L)}{\langle[D]\rangle}\,.
\]

When $h$ has even rank $2n$ and trivial discriminant, the
\emph{Clifford invariant} $\mathscr{C}\ell(h)$ of $h$ is defined as
\[
\mathscr{C}\ell(h):=\mathscr{C}\ell_{H_{2n}}(h)\;\in\;\frac{{}_2\Br(L)}{\langle[D]\rangle}\,,
\]where $H_{2n}$ denotes a hyperbolic hermitian form of rank $2n=\rank(h)$ over
$(D,\,\sigma)$. If $D=L$ and $h=q$ is a nonsingular quadratic form
over $L$, then $\mathscr{C}\ell(h)$ coincides with the usual Clifford
invariant of the quadratic form $q$.

\

\newpara\label{para2p11NEW} Let $(D,\,\sigma)$ be a central division
algebra with an orthogonal involution over $L$. We denote by $\bfU_{2n}(D,\,\sigma)$,
$\bfSU_{2n}(D,\,\sigma)$ and $\bfSpin_{2n}(A,\,\sigma)$ respectively
the unitary group, the special unitary group and the spin group of
the hyperbolic form over $(D,\,\sigma)$ defined by the matrix
$H_{2n}=\smatr{0}{I_n}{I_n}{0}$.

Let $h$ be a
hermitian form of even rank $2n$, trivial discriminant and trivial
Clifford invariant. 
There is an element $\xi\in H^1(L,\,\bfSpin_{2n}(D,\,\sigma))$ which
is mapped to the class $[h]\in H^1(L,\,\bfU_{2n}(D,\,\sigma))$ under
the composite map
\[
H^1(L,\,\bfSpin_{2n}(D,\,\sigma))\lra
H^1(L\,,\,\bfSU_{2n}(D,\,\sigma))\lra
H^1(L\,,\,\bfU_{2n}(D,\,\sigma))\,.
\]Let
\[
R_{\bfSpin_{2n}(D,\,\sigma)}\,:\;\;
H^1(L\,,\,\bfSpin_{2n}(D,\,\sigma))\lra H^3(L\,,\,\mathbb{Q}/\mathbb{Z}(2))
\]be the usual Rost invariant map of the simply connected group
$\bfSpin_{2n}(D,\,\sigma)$ (cf. \cite[$\S$31.B]{KMRT}). It is shown
in \cite[p.664]{BP2} that the class of
$R_{\bfSpin_{2n}(D,\,\sigma)}(\xi)$ in the quotient
\[
\frac{H^3(L\,,\,\mathbb{Q}/\mathbb{Z}(2))}{H^1(L\,,\,\mu_2)\cup (D)}
\] is well-defined. The \emph{Rost invariant} $\sR(h)$ of the form $h$ is
defined as
\[
\sR(h):=[R_{\bfSpin_{2n}(D,\,\sigma)}(\xi)]\,\in\;
\frac{H^3(L\,,\,\mathbb{Q}/\mathbb{Z}(2))}{H^1(L\,,\,\mu_2)\cup (D)}\,.
\]

\newpara\label{para2p12NEW} Let $(D,\,\sigma)$ be a quaternion
algebra with an orthogonal involution over $L$.  We will need some
further analysis on the map $\wt{\rho}: W(\wt{D}\,,\,\tau_1)\to
W(D,\,\sigma)$ in the exact sequence \eqref{eq2p7p2NEW}. Note that
in this case $\wt{D}=M$ is a quadratic field extension of $L$ and
$\tau_1$ is the nontrivial element $\iota$ of the Galois group
$\Gal(M/L)$. Let $\bfU_{2n}(M,\,\iota)$ and $\bfSU_{2n}(M,\,\iota)$
denote the unitary group and the special unitary group of the
hyperbolic form over $(M,\,\iota)$ defined by the matrix
$H_{2n}=\matr{0}{I_n}{I_n}{0}$. We have
\[
\bfU_{2n}(M,\,\iota)(L)=\{A\in
\rM_{2n}(M)\,|\,A.H_{2n}\iota(A)^t=H_{2n}\}\,.
\]Note that for $A\in\rM_{2n}(M)$, $\iota(A)=\Int(\mu)\circ \sigma(A)=\mu
A\mu^{-1}$ (cf. (\ref{para2p7NEW})) and
\[
A.H_{2n}.\iota(A)^t=H_{2n}\iff
(A.H_{2n}.\iota(A)^t)^t=(H_{2n})^t\iff \iota(A).H_{2n}.A^t=H_{2n}\,.
\]Therefore, for $A\in\bfU_{2n}(M,\,\iota)(L)$, we have
\[\begin{split}
A.\mu^{-1}\lambda H_{2n}.\sigma(A)^t&=A.\mu^{-1}\lambda H_{2n}.
A^t\\
&=\mu^{-1}(\mu A\mu^{-1})\lambda H_{2n}. A^t\\
&=\mu^{-1}\lambda \iota(A).H_{2n}A^t=\mu^{-1}\lambda H_{2n}
\end{split}
\]inside $\rM_{2n}(D)$. So we have a natural inclusion
\[
\bfU_{2n}(M,\,\iota)(L)\subseteq \bfU(\mu^{-1}\lambda
H_{2n})(L)=\{B\in \rM_{2n}(D)\,|\, B.\mu^{-1}\lambda H_{2n}.
\sigma(B)^t=\mu^{-1}\lambda H_{2n}\}\,.
\]In fact, this defines an inclusion of algebraic groups over $L$:
\[
\rho'\,:\;\;\bfU_{2n}(M,\,\iota)\lra \bfU(\mu^{-1}\lambda
H_{2n})\,;\quad A\longmapsto A\,.
\]
By \cite[p.402, Example$\;$29.19]{KMRT}, any element $\xi$ of
$H^1(L,\,\bfU_{2n}(M,\,\iota))$ is represented by a matrix $S\in
\GL_{2n}(M)$ which is symmetric with respect to the adjoint
involution $\iota_{H_{2n}}$ on $\rM_{2n}(M)$, and $\xi$ is the
isomorphism class of the hermitian form $H_{2n}S^{-1}$. The natural
map
\[
H^1(L\,,\,\bfU_{2n}(M,\,\iota))\lra H^1(L\,,\,\bfU(\mu^{-1}\lambda
H_{2n}))
\]induced by the homomorphism $\rho'$ maps $\xi$ to the class of the
hermitian form $\mu^{-1}\lambda H_{2n}S^{-1}$. On the other hand, by
the construction of the homomorphism $\wt{\rho}: W(M,\,\iota)\to
W(D,\,\sigma)$, the form $H_{2n}S'$ over $(M,\,\iota)$ is mapped to
the form $\mu^{-1}\lambda H_{2n}S^{-1}$ over $(D,\,\sigma)$. Hence
the natural map
\[
H^1(L,\,\bfU_{2n}(M,\,\iota))\lra H^1(L\,,\,\bfU(\mu^{-1}\lambda
H_{2n}))
\]is compatible with the restriction of $\rho$ to forms of rank
$2n$.

Clearly, the inclusion $\rho': \bfU_{2n}(M,\,\iota)\to
\bfU(\mu^{-1}\lambda H_{2n})$ induces an inclusion
$\bfSU_{2n}(M,\,\iota)\to \bfSU(\mu^{-1}\lambda H_{2n})$ (cf.
\cite[p.671]{BP2}). A choice of isomorphism of hermitian forms
$\mu^{-1}\lambda H_{2n}\cong H_{2n}$ over $(D,\,\sigma)$ yields an
injection
\[
\bfSU_{2n}(M,\,\iota)\lra \bfSU(H_{2n})=\bfSU_{2n}(D,\,\sigma)\,.
\]This lifts to a homomorphism
\[
\rho_0\,:\;\;\bfSU_{2n}(M,\,\iota)\lra \bfSpin_{2n}(D,\,\sigma)\,.
\]The composition
\[
\bfSU_{2n}(M,\,\iota)\overset{\rho_0}{\lra}\bfSpin_{2n}(D,\,\sigma)\lra
\bfU_{2n}(D,\,\sigma)
\]induces a commutative diagram
\[
\xymatrix{ H^1(L\,,\,\bfSU_{2n}(M\,,\,\iota)) \ar[dr]_{\rho'}
\ar[rr]^{\rho_0} && H^1(L\,,\,\bfSpin_{2n}(D\,,\,\sigma)) \ar[dl]_{} \\
& H^1(L\,,\,\bfU_{2n}(D\,,\,\sigma)) & }
\]such that the map $\wt{\rho}\,:\;W(M\,,\,\iota)\to W(D\,,\,\sigma)$ restricted to forms of
rank $2n$ and of trivial discriminant is compatible with the map
$\rho'$ at the level of cohomology sets. Moreover, for any $\xi\in
H^1(L,\,\bfSU_{2n}(M,\,\iota))$, one has  by
\cite[Prop.$\;$3.20]{BP2}
\[
R_{\bfSpin_{2n}(D\,,\,\sigma)}(\rho_0(\xi))=R_{\bfSU_{2n}(M\,,\,\iota)}(\xi)\,\in\;
H^3(L\,,\,\mathbb{Q}/\mathbb{Z}(2))\,,
\]i.e., $\rho_0(\xi)\in H^1(L\,,\,\bfSpin_{2n}(D,\,\sigma))$ has the
same Rost invariant as $\xi$. If $h$ is a hermitian form over
$(D,\,\sigma)$ representing the class $\rho'(\xi)\in
H^1(L,\,\bfU_{2n}(D,\,\sigma))$, then the Rost invariant of the form
$h$ is
\[
\sR(h)=[R_{\bfSpin_{2n}(D\,,\,\sigma)}(\rho_0(\xi))]=
[R_{\bfSU_{2n}(M\,,\,\iota)}(\xi)]\,\in\;\frac{H^3(L\,,\,\mathbb{Q}/\mathbb{Z}(2))}{H^1(L\,,\,\mu_2)\cup
(D)}
\]by definition (cf. (\ref{para2p11NEW})).

\

\newpara\label{para3p1TEMP} We shall also use the notion of Rost invariant of hermitian forms over an algebra with unitary involution. The definition is as follows. Let $E$ be a field of characteristic $\neq 2$, $L/E$ a quadratic field extension and $(D,\,\tau)$ a central division algebra over $L$
with a unitary $L/E$-involution.  Let $\bfU_{2n}(D,\,\tau)$ and $\bfSU_{2n}(D,\,\tau)$ denote respectively the unitary group and the special unitary group of the hyperbolic form $\smatr{0}{I_n}{I_n}{0}$ over $(D,\,\tau)$. For a hermitian form $h$ of rank $2n$ and trivial discriminant over $(D,\,\tau)$, we may define its \emph{Rost invariant} $\mathscr{R}(h)$ by
\[
\mathscr{R}(h):=[R_{\bfSU_{2n}(D,\,\tau)}(\xi)]\,\in\; \frac{H^3(E,\,\mathbb{Q}/\mathbb{Z}(2))}{\mathrm{Cores}_{L/E}((L^{*1})\cup (D))}\,,
\]where $\xi\in H^1(E,\,\bfSU_{2n}(D,\,\tau))$ is any lifting of the class $[h]\in H^1(E,\,\bfU_{2n}(D,\,\tau))$ and
\[
L^{*1}=(R^1_{L/E}\mathbb{G}_m)(E)=\{ a\in L^*\,|\, N_{L/E}(a)=1 \}\,.
\]Indeed, by \cite[Appendix, Remark$\;$B]{PaPr}, the class $[R_{\bfSU_{2n}(D,\,\tau)}(\xi)]$ is independent of the choice of the lifting $\xi$, so that this
Rost invariant $\mathscr{R}(h)$ is well defined. Note that if $D=D_0\otimes_EL$ for some central division algebra $D_0$ over $E$, then
\[
\mathrm{Cores}_{L/E}((L^{*1})\cup (D))=0
\]and hence the Rost invariant of $h$ is simply the usual Rost invariant of any lifting $\xi\in H^1(E,\,\bfSU_{2n}(D,\,\tau))$
of the isomorphism class of $h$.

\subsection{Spinor norms}

\newpara\label{para4p1temp} Let $E$ be a field of characteristic
different from 2, $A$ a central simple algebra over $E$ and $\sigma$
an orthogonal involution on $A$. Let $h$ be a nonsingular hermitian
form over $(A\,,\,\sigma)$. The exact sequence of algebraic groups
\[
1\lra \mu_2\lra \bfSpin(h)\lra \bfSU(h)\lra 1\,,
\]induces a connecting map
\[
\delta\,:\;\;\bfSU(h)(E)\lra H^1(E\,,\,\mu_2)=E^*/E^{*2}
\]which we call the \emph{spinor norm} map. We will write
\[
\Sn(h_E):=\mathrm{Im}\left(\delta\,:\;\bfSU(h)(E)\lra E^*/E^{*2}\right)
\]for the image of the above spinor norm map. If $A=E$, $\sigma=\mathrm{id}$ and $h=q$ is a quadratic form,
the spinor norm map $\delta: \bfSO(q)(E)\to E^*/E^{*2}$ has an
explicit description as follows (cf. \cite[p.108]{Lam}): Any element
$\theta\in \bfSO(q)(E)$ can be written as the product of an even
number of hyperplan reflections associated with anisotropic vectors
$v_1,\dotsc, v_{2r}$. The spinor norm $\delta(\theta)$ is equal to
the class of the product $q(v_1)\cdots  q(v_{2r})$ in $E^*/E^{*2}$.

\

A deep theorem of Merkurjev is the following norm principle for spinor norms.

\begin{thm}[{Merkurjev, \cite[6.2]{Mer96}}]\label{thm4p2temp}
With notation as in $(\ref{para4p1temp})$, assume that
$\deg(A).\rank(h)$ is even and at least $4$.

Then the image $\Sn(h_E)$ of the spinor norm map is equal to the
subgroup of $E^*/E^{*2}$ generated by the canonical images of the
norm groups $N_{L/E}(L^*)$ over all finite field extensions $L/E$
such that $A_L$ is split and $h_L$ is isotropic.
\end{thm}

The following  corollary is immediate from the above theorem.

\begin{coro}\label{coro2p14NEW}
With notation and hypotheses as in Theorem$\;\ref{thm4p2temp}$, for
any finite field extension $E'/E$, one has
\[
N_{E'/E}(\Sn(h_{E'}))\subseteq \Sn(h_E)\,.
\]
\end{coro}

\newpara\label{para2p15NEW} With notation and hypotheses as in Theorem$\;\ref{thm4p2temp}$, the well-known norm principle for reduced norms states that
the subgroup $\Nrd(A^*)\subseteq E^*$ of reduced norms is generated
by the norm groups $N_{L/E}(L^*)$, where $L/E$ runs over all finite
field extensions such that $A_L$ is split. So
Theorem$\;$\ref{thm4p2temp} implies that $\Sn(h_E)$ is contained in
the canonical image of $\Nrd(A^*)$ in $E^*/E^{*2}$.

\

\newpara\label{para4p3temp}
Let $(A,\,\sigma)$ be a central simple algebra with an orthogonal
involution over a field $E$ of characteristic $\neq 2$. Let $L/E$ be
a field extension which splits $A$ and let $\phi:
(A\,,\,\sigma)\otimes_EL\cong (\rM_n(L)\,,\,\sigma_{q_0})$ be an
isomorphism of $L$-algebras with involution, where $\sigma_{q_0}$ is
the adjoint involution of a quadratic form $q_0$ of rank $n=\deg(A)$
over $L$. Let $h$ be a hermitian form over
$(A\,,\,\sigma)\otimes_EL$. Then by Morita theory (cf.
(\ref{para2p6NEW})), $h$ corresponds via the above isomorphism
$\phi$ to a quadratic form $q$ of rank $n.\rank(h)=\deg(A).\rank(h)$
over $L$. The similarity class $[q]\in W(L)$ of $q$ is uniquely
determined by $h$ and is independent of the choice of $\phi$ and
$q_0$. The hermitian form $h_L$ is isotropic if and only if the
quadratic form $q_L$ is isotropic. So, if $\deg(A).\rank(h)$ is even
and at least $4$, one has $\Sn(q_L)=\Sn(h_L)$ by
Theorem$\;$\ref{thm4p2temp}.

\section{Some easy cases}\label{sec3}

We shall now start the proofs of our main theorems. In a few cases, as
may be already well-known to specialists, the results basically
follow by combining a general injectivity result for the Rost
invariant and a Hasse principle coming from higher dimensional class
field theory.

\

\newpara\label{para3p1NEW} Recall that our base field $K$ is the function field of a
$p$-adic arithmetic surface or a local henselian surface with finite
residue field (cf. (\ref{para1p2NEWTEMP})). Namely, $K$ is either

(the case of $p$-adic arithmetic surface) the function field $F(C)$
of a smooth projective geometrically integral curve $C$ over $F$,
where $F$ is a $p$-adic field with ring of integers $A$ and residue
field $k$;

or

(the case of local henselian surface)  the field of fractions
$\mathrm{Frac}(A)$ of a 2-dimensional, henselian, excellent local
domain $A$ with finite residue field $k$ of characteristic $p$.

In either case, by abuse of language we say $k$ is the residue field
of $K$ and $p=\mathrm{char}(k)$ is the \emph{residue characteristic}
of $K$.

In our proofs of the main theorems, we only use local conditions at
\emph{divisorial valuations}, i.e., valuations corresponding to codimension
1 points of regular proper models (cf. (\ref{para1p2NEWTEMP})). More
precisely, the set $\Omega_A$ of divisorial valuations of the field
$K$ is the subset of $\Omega_K$ defined as follows:

In $p$-adic arithmetic case, define
\[
\Omega_A=\bigcup_{\mathcal{X}\to\mathrm{Spec} A}\mathcal{X}^{(1)}\,,
\]where $\mathcal{X}\to \mathrm{Spec} A$ runs over proper flat morphisms from a
regular integral scheme $\mathcal{X}$ with function field $K$ and
$\mathcal{X}^{(1)}$ denotes the set of codimension 1 points of
$\mathcal{X}$ identified with a subset of $\Omega_K$.

In the local henselian case, define
\[
\Omega_A=\bigcup_{\mathcal{X}\to\mathrm{Spec} A}\mathcal{X}^{(1)}\,,
\]
where $\mathcal{X}\to \mathrm{Spec} A$ runs over proper birational morphisms
from a regular integral scheme $\mathcal{X}$ with function field $K$
and $\mathcal{X}^{(1)}$ denotes the set of codimension 1 points of
$\mathcal{X}$ identified with a subset of $\Omega_K$.

\

\newpara\label{para3p2NEW} Let $L/K$ be a finite field extension.
Then $L$ is a field of the same type as $K$ if $K$ is the function
field of a $p$-adic arithemetic surface or a local henselian surface
with finite residue field. In the $p$-adic arithmetic case, let $F'$
be the field of constants of $L$ and let $A'$ be the integral
closure of $A$ in $F'$. In the local henselian case, let $A'$ be the
integral closure of $A$ in $L$. Then the set $\Omega_{A'}$ of
divisorial discrete valuations of $L$ is precisely the set of
discrete valuations $w\in\Omega_L$ lying over valuations in
$\Omega_A\subseteq \Omega_K$.

\

\newpara\label{para3p3NEW} By the general theory of semisimple groups
(see e.g. \cite[p.365, Thm.$\;$26.8]{KMRT}), any semisimple simply
connected group $G$ over $K$ is a finite product of groups of the
form $R_{L/K}(G')$, where $L/K$ is a finite separable field
extension, $G'$ is an absolutely simple simply connected group over
$L$ and $R_{L/K}$ denotes the Weil restriction functor. For each
$v\in\Omega_A$, one has $L\otimes_KK_v\cong\prod_{w\,|\,v}L_w$ and
by Shapiro's lemma,
\[
H^1(K,\,R_{L/K}G')\cong H^1(L\,,\,G')\quad\text{ and }\quad
H^1(K_v,\,R_{L/K}G')\cong\prod_{w\,|\,v}H^1(L_w\,,\,G')\,.
\]
Therefore, to prove the Hasse principle for semisimpe simply
connected groups we may reduce to the case where $G$ is an
absolutely simple simply connected group.

\subsection{The quasi-split case}

We recall the proof of the Hasse principle for quasi-split groups
without $E_8$ factors (cf. \cite[Thm.$\;$5.4]{CTPaSu}).

\

The  following theorem is of particular importance to us.

\begin{thm}\label{thm1p3temp}
Let $K$ be the function field of a $p$-adic arithmetic surface or a
local henselian surface with finite residue field of characteristic
$p$. Let $\Omega_A$ be the set of divisorial discrete valuations of
$K\,($as defined in $(\ref{para3p1NEW}))$.

$(\oi)$ $($Kato, \cite{Ka86}$)$ In the $p$-adic arithmetic case, the
natural map
\[
H^3(K\,,\,\mathbb{Q}/\mathbb{Z}(2))\lra\prod_{v\in\Omega_A}H^3(K_v\,,\,\mathbb{Q}/\mathbb{Z}(2))
\]is injective.

$(\ii)$ $($Saito, \cite{Sai87}, cf. \cite[Prop.$\;$4.1]{Hu11}$)$ In
the local henselian case, let $n>0$ be an integer prime to $p$. Then
the natural map
\[
H^3(K\,,\,\mu_n^{\otimes 2})\lra
\prod_{v\in\Omega_A}H^3(K_v\,,\,\mu_n^{\otimes 2})
\]is injective.
\end{thm}

The next result is an injectivity statement for the Rost invariant
of quasi-split groups.

\begin{thm}[{cf. \cite[Thm.$\;$5.3]{CTPaSu}}]\label{thm3p5NEW}\footnote{See \cite{Preeti} for a recent improvement of this theorem.}
Let $E$ be a field of cohomological $2$-dimension $\le 3$ and let
$G$ be an absolutely simple simply connected quasi-split group over
$E$. Assume that $G$ is not of type $E_8$. Assume further the
characteristic of $E$ is not $2$ if $G$ is of classical type $B_n$
or $D_n$.

Then the kernel of the Rost invariant map $R_G: H^1(E,\,G)\to
H^3(E,\,\mathbb{Q}/\mathbb{Z}(2))$ is trivial.
\end{thm}
\begin{proof}
For a quasi-split group of type ${}^1A_n$ or $C_n$, it is well-known
that $H^1(E,\,G)=1$ over an arbitrary field $E$. For exceptional
groups (not of type $E_8$), the kernel of the Rost invariant is
trivial over an arbitrary field by the work of Chernousov, Garibaldi
and Gille (cf. \cite[Thm.$\;$5.2]{Gille10}, \cite{Chern03},
\cite{Gari01} and \cite{Gille00}). If $G$ is of type ${}^2A_n$,
$B_n$ or classical type $D_n$, the proof can be done as in
\cite[Thm.$\;$5.3]{CTPaSu}, by passing to a quadratic form argument.
\end{proof}

The $p$-adic case of the following result is \cite[Thm.$\;$5.4]{CTPaSu}.

\begin{thm}\label{thm3p6NEW}
Let $K$ be the function field of a $p$-adic arithmetic surface or a
local henselian surface with finite residue field of characteristic
$p$. Let $G$ be an absolutely simple simply connected quasi-split
group not of type $E_8$ over $K$. Assume $p\notin S(G)$ in the local henselian case
$($see $(\ref{para1p5NEWTEMP})$ for the definition of $S(G)\,)$.

Then the natural map
\[ H^1(K\,,\,G)\lra
\prod_{v\in\Omega_A}H^1(K_v\,,\,G)
\]has a trivial kernel.
\end{thm}
\begin{proof}
The result follows from the following commutative diagram
\[
\begin{CD}
H^1(K\,,\,G) @>>> \prod_{v\in\Omega_A}H^1(K_v,\,G) \\
@VVV @VVV\\
H^3(K\,,\,\mu_n^{\otimes 2}) @>>>
\prod_{v\in\Omega_A}H^3(K_v,\,\mu_n^{\otimes 2})
\end{CD}
\]where the vertical maps have trivial kernel by Theorem$\;$\ref{thm3p5NEW} and the bottom horizontal map is
injective by Theorem$\;$\ref{thm1p3temp}.
\end{proof}

\subsection{Groups of type ${}^1A_n^*$}

For groups of inner type $A_n^*$, the proof is essentially the same
as the quasi-split case.

\begin{thm}\label{thm1p4temp}
Let $K$ be the function field of a $p$-adic arithmetic surface or a
local henselian surface with finite residue field of characteristic
$p$. Let $A$ a central simple $K$-algebra of square-free index $n$
and $G=\bfSL_1(A)$. Assume $p\nmid n$ in the local henselian case.

Then the natural map
\[
H^1(K,\,G)\lra \prod_{v\in\Omega_A}H^1(K_v\,,\,G)
\]is injective.
\end{thm}
\begin{proof}A well-known theorem of Suslin (\cite[Thm.$\;$24.4]{Suslin85}) implies that under the assumptions of the theorem,
the Rost invariant map
\[
H^1(E\,,\,\bfSL_1(A))=E^*/\Nrd(A^*)\lra H^3(E\,,\,\mu_n^{\otimes
2})\;;\quad \lambda\longmapsto (\lambda)\cup (A)
\]is injective for $E=K$ or $K_v$. An argument similar to the proof of
Theorem$\;$\ref{thm3p6NEW} yields the result.
\end{proof}

\subsection{Groups of type $C_n^*$}

\begin{lemma}\label{lemma2p1temp}
Let $K$ be the function field of a $p$-adic arithmetic surface or a
local henselian surface with finite residue field of characteristic
$p$. Assume $p\neq 2$ in the local henselian case.

Then the natural map
\[
I^3(K)\lra \prod_{v\in\Omega_A}I^3(K_v)
\]is injective.
\end{lemma}
\begin{proof}
Consider the following commutative diagram
\[
\begin{CD}
I^3(K) @>>> \prod_{v\in\Omega_A}I^3(K_v)\\
@V{e_3}VV @VVV\\
H^3(K\,,\,\mathbb{Z}/2) @>>> \prod_{v\in\Omega_A}H^3(K_v\,,\,\mathbb{Z}/2)
\end{CD}
\]where the vertical maps are induced by the Arason invariants.
Since $\mathrm{cd}_2(K)\le 3$, we have $I^4(K)=0$. So the map
\[
e_3\,:\;I^3(K)\lra H^3(K\,,\,\mathbb{Z}/2)
\]is injective by \cite[Prop.$\;$3.1]{AEJ86}. The map
\[
H^3(K\,,\,\mathbb{Z}/2)\lra \prod_{v\in\Omega_A}H^3(K_v\,,\,\mathbb{Z}/2)
\]is injective by Theorem$\;$\ref{thm1p3temp}. The lemma then follows
from the above commutative diagram.
\end{proof}

\begin{thm}\label{thm2p2temp}Let $K$ be the function field of a $p$-adic arithmetic surface or a
local henselian surface with finite residue field of characteristic
$p$. Let $D$ be a quaternion division algebra over $K$ with standard
involution $\tau_0$ and $h$ a nonsingular hermitian form over
$(D,\,\tau_0)$. Assume $p\neq 2$ in
the local henselian case. Let $G=\bfU(h)$ be the unitary group of
the hermitian form $h$.

Then the natural map
\[
H^1(K\,,\,G)\lra \prod_{v\in\Omega_A}H^1(K_v\,,\,G)
\]is injective.
\end{thm}
\begin{proof}
The pointed set $H^1(K\,,\,G)=H^1(K\,,\,\bfU(h))$ classifies up to
isomorphism hermitian forms over $(D,\,\tau_0)$ of the same rank as
$h$. Let $h_1$ and $h_2$ be hermitian forms over $(D\,,\,\tau_0)$ of
the same rank as $h$. Put $h'=h_1\bot (-h_2)$. Note that $h'$ has
even rank, so the class of $q_{h'}$ in the Witt group $W(K)$ lies in
the subgroup $I^3(K)=I(K)\cdot I^2(K)$ (cf. (\ref{para2p3NEW})).
Thus
\[
[q_{h_1}]-[q_{h_2}]=[q_{h'}]\,\in \,I^3(K)\,.
\]If $(h_1)_v\cong (h_2)_v$ for all $v\in\Omega_A$, then by
Lemma$\;$\ref{lemma2p1temp}, $[q_{h'}]=0\in I^3(K)$. This implies
that $q_{h_1}\cong q_{h_2}$ over $K$. Two hermitian forms over
$(D\,,\,\tau_0)$ are isomorphic if and only if their trace forms are
isomorphic as quadratic forms (cf. (\ref{para2p3NEW})). So we get
from the above that $h_1\cong h_2$, proving the theorem.
\end{proof}

\subsection{Groups of type $G_2$ or $F_4^{red}$}

\begin{thm}\label{thm6p1temp}
Let $K$ be the function field of a $p$-adic arithmetic surface or a
local henselian surface with finite residue field of characteristic
$p$. Let $G$ be an absolutely simple  simply connected
 group of type $G_2$ over $K$. Assume $p\neq 2$ in the local
 henselian case.

 Then the natural map
\[
H^1(K\,,\,G)\lra \prod_{v\in\Omega_A}H^1(K_v\,,\,G)
\]has a trivial kernel.
\end{thm}
\begin{proof}
The group $G$ is isomorphic to $\mathbf{Aut}_{alg}(C)$ for some
Cayley algebra $C$ over $K$. Let $\xi\in H^1(K\,,\,G)$ be a locally
trivial class and let $C'$ be a Cayley algebra which represents
$\xi$. We have $C_{K_v}\cong C'_{K_v}$ for every $v\in\Omega_A$ by
hypothesis and we want to show $C\cong C'$ over $K$. Since two
Cayley algebras are isomorphic if and only if their norm forms are
isomorphic and since the norm form of a Cayley algebra is a 3-fold
Pfister form (cf. \cite[p.460]{KMRT}), the result follows easily
from Lemma$\;$\ref{lemma2p1temp}.
\end{proof}

\begin{thm}\label{thm6p2temp}
Let $K$ be the function field of a $p$-adic arithmetic surface or a
local henselian surface with finite residue field of characteristic
$p$. Assume $p\nmid 6$ in the local henselian case. Let
$G=\mathbf{Aut}_{alg}(J)$ be the automorphism group of a
\emph{reduced} $27$-dimensional exceptional Jordan algebra over $K$.

Then the natural map
\[
H^1(K\,,\,G)\lra \prod_{v\in\Omega_A}H^1(K_v\,,\,G)
\]has a trivial kernel.
\end{thm}
\begin{proof}
Recall that (cf. \cite[$\S$9]{Ser94}) to each exceptional Jordan
algebra $J'$ of dimension 27 over a field $F$ of characteristic not
2 or 3, one can associate three invariants
\[
f_3(J')\in H^3(F\,,\,\mathbb{Z}/2)\,,\;\;f_5(J')\in
H^5(F\,,\,\mathbb{Z}/2)\quad\text{and }\quad g_3(J')\in H^3(F\,,\,\mathbb{Z}/3)\,.
\]One has $g_3(J')=0$ if and only if $J'$ is reduced. Two reduced exceptional Jordan algebras are isomorphic if and only if their
$f_3$ and $f_5$ invariants are the same.

Now our base field $K$ has cohomological 2-dimension
$\mathrm{cd}_2(K)=3$. So the invariant $f_5(J')$ is always zero. Let
$\xi\in H^1(K\,,\,G)$ correspond to the isomorphism class of an
exceptional Jordan algebra $J'$ over $K$. Assume that $\xi$ is
locally trivial in $H^1(K_v\,,\,G)$ for every $v\in\Omega_A$. By
Theorem$\;$\ref{thm1p3temp}, we have $f_3(J)=f_3(J')$ and
$g_3(J)=g_3(J')$. Since $J$ is reduced by assumption, we have
$g_3(J')=0$ and hence $J'$ is reduced. Thus it follows that $J\cong
J'$ over $K$, showing that $\xi$ is trivial in $H^1(K\,,\,G)$ as
desired.
\end{proof}

\section{Spin groups of quadratic forms}\label{sec4}

\newpara\label{para3p1temp} Let $E$ be a field of characteristic
different from 2 and $q$ a nonsingular quadratic form of rank $\ge
3$ over $E$. Recall that $\Sn(q_E)$ denotes the image of the spinor
norm map
\[
\bfSO(q)(E)\lra E^*/E^{*2}\,,\,
\]i.e., the connecting map associated to the cohomology of the exact sequence
\[
1\lra \mu_2\lra \bfSpin(q)\lra \bfSO(q)\lra 1\,.
\]

\begin{prop}\label{prop3p2temp}Let $K$ be the function field of a $p$-adic arithmetic surface or a
local henselian surface with finite residue field of characteristic
$p$. Assume $p\neq 2$ in the local henselian case. Let $q$ be a
nonsingular quadratic form of rank $3$ or $4$ over $K$.

Then the natural map
\[
\frac{K^*/K^{*2}}{\Sn(q_K)}\lra
\prod_{v\in\Omega_A}\frac{K_v^*/K_v^{*2}}{\Sn(q_{K_v})}
\]is injective.
\end{prop}
\begin{proof}
If $\rank(q)=3$, we may assume $q=\langle 1\,,\,a\,,\,b\rangle$
after scaling. Let $D$ be the quaternion algebra $(-a\,,\,-b)_K$
over $K$. Then $\Sn(q)=\Nrd(D^*)$ modulo squares. The result then
follows from Theorem$\;$\ref{thm1p4temp}.

Assume next $\rank(q)=4$. If $\disc(q)=1$, we may assume after
scaling $q=\langle 1\,,\,a\,,\,b\,,\,ab\rangle$. Put
$D=(-a\,,\,-b)_K$. Then $\Sn(q)=\Nrd(D^*)$ and the result follows
again from Theorem$\;$\ref{thm1p4temp}. If $d=\disc(q)$ is
nontrivial in $K^*/K^{*2}$, we may assume $q=\langle
1\,,\,a\,,\,b\,,\,abd\rangle$. Then
\[
\Sn(q_K)=\Nrd\left(D_{K(\sqrt{d})}^*\right)\cap K^*\quad\text{modulo
squares}
\]by \cite[p.214, Coro.$\;$15.11]{KMRT}. The field $K(\sqrt{d})$ is
a field of the same type as $K$ (cf. (\ref{para3p2NEW})). Let
$\Omega_{A'}$ denote the set of divisorial valuations of
$K'=K(\sqrt{d})$. If $\al\in K^*$ lies in $\Sn(q_{K_v})$ for all
$v\in\Omega_A$, then $\al$ is a reduced norm from $D_{K'_w}$ for all
$w\in\Omega_{A'}$. By Theorem$\;$\ref{thm1p4temp}, $\al$ is a
reduced norm from $D_{K'}=D_{K(\sqrt{d})}$. This finishes the proof.
\end{proof}

Recall that the $u$-invariant $u(E)$ of a field $E$ of characteristic $\neq 2$ is the supremum of dimensions of anisotropic
quadratic forms over $E$ (so $u(E)=\infty$, if such dimensions can be arbitrarily large).

\begin{prop}\label{prop3p3temp}
Let $E$ be a field of characteristic $\neq 2$ and $q$ a nonsingular
quadratic form of rank $r$ over $E$. Assume $u(E)<2r$.

Then $\Sn(q_E)=E^*/E^{*2}$, i.e., the spinor norm map
\[
\bfSO(q)(E)\lra E^*/E^{*2}
\] is surjective.
\end{prop}
\begin{proof}
The image $\Sn(q_E)$ of the spinor norm map consists of elements of
the form $\prod^{2m}_{i}q(v_i)$, where $v_i$ are anisotropic vectors
for $q$ (cf. (\ref{para4p1temp})). If $q$ is isotropic over $E$, then for every $\al\in E^*$,
there is a vector $v_{\al}$ such that $q(v_{\al})=\al$. Let $v_1$ be
a vector such that $q(v_1)=1$. Then we have
$\al=q(v_{\al}).q(v_1)\in\Sn(q_E)$.

Assume next $q$ is anisotropic. For any $\al\in E^*$, the form
$q\bot (-\al.q)$ is isotropic over $E$ by the assumption on the
$u$-invariant. Hence there are vectors $x,\,y$ such that
$q(x)-\al.q(y)=0$. Since $q$ is anisotropic, we have
$\lambda:=q(y)\in E^*$ and $q(x)\in E^*$. It follows that
\[
\al=q(x).q(y)^{-1}=\lambda^{-2}q(x).q(y)=q(x).q(\lambda^{-1}y)\,\in
\,\Sn(q_E)\,
\]whence the desired result.
\end{proof}

\begin{coro}\label{coro3p4temp}Let $K$ be the function field of a $p$-adic arithmetic surface or a
local henselian surface with finite residue field of characteristic
$p$. Assume $p\neq 2$ in the local henselian case. Let $q$ be a
nonsingular quadratic form of rank $\ge 5$ over $K$.

Then $\Sn(q_K)=K^*/K^{*2}$, i.e., the spinor norm map
\[
\bfSO(q)(K)\lra K^*/K^{*2}
\]
 is surjective.
\end{coro}
\begin{proof}In the $p$-adic arithmetic case, we have $u(K)=8$ by \cite{PaSu} (if $p\neq 2$) or
\cite{Le10} (see also \cite{HHK}). In the local henselian case, it
is proved in \cite[Thm.$\;$1.2]{Hu11} that $u(K)=8$. The result then
follows immediately from Proposition$\;$\ref{prop3p3temp}.
\end{proof}

\begin{thm}\label{thm3p5temp}Let $K$ be the function field of a $p$-adic arithmetic surface or a
local henselian surface with finite residue field of characteristic
$p$. Assume $p\neq 2$ in the local henselian case. Let $q$ be a
nonsingular quadratic form of rank $\ge 3$ over $K$ and
$G=\bfSpin(q)$.

$(\oi)$ The natural map
\[
H^1(K\,,\,G)\lra\prod_{v\in\Omega_A}H^1(K_v\,,\,G)
\]has a trivial kernel.

$(\ii)$ The Rost invariant
\[
R_G\,:\;H^1(K\,,\,G)\lra H^3(K\,,\,\mathbb{Q}/\mathbb{Z}(2))
\]has a trivial kernel if $\rank(q)\ge 5$.
\end{thm}
\begin{proof}
Consider the exact sequence of algebraic groups
\[
1\lra \mu_2\lra \bfSpin(q)=G\lra \bfSO(q)\lra 1
\]which gives rise to an exact sequence of pointed sets
\begin{equation}\label{eq3p5p1temp}
\bfSO(q)(K)\overset{\delta}{\lra}K^*/K^{*2}\overset{\psi}{\lra}
H^1(K\,,\,\bfSpin(q))\overset{\eta}{\lra} H^1(K\,,\,\bfSO(q))\,.
\end{equation}The image of the map $\eta$ is in bijection with
isomorphism classes of nonsingular quadratic forms $q'$ with the
same rank, discriminant and Clifford invariant as $q$.

Let $\xi\in H^1(K\,,\,G)=H^1(K\,,\,\bfSpin(q))$ with $\eta(\xi)\in
H^1(K\,,\,\bfSO(q))$ corresponding to a quadratic form $q'$. Then in
the Witt group $W(K)$ the class of $q\bot (-q')$ lies in $I^3(K)$ by
Merkurjev's theorem (cf. \cite[p.89, Thm.$\;$2.14.3]{Schar}) and its
Arason invariant $e_3([q\bot (-q')])\in H^3(K\,,\,\mathbb{Z}/2)$ coincides
with Rost invariant $R_G(\xi)$ of $\xi$ when $\rank(q)\ge 5$ (\cite[p.437]{KMRT}).

For (i), assume the canonical image $\xi_v$ of $\xi$ in
$H^1(K_v\,,\,G)$ is trivial for every $v\in\Omega_A$. We have
\[
[q\bot
(-q')]_v=0\,\in\,I^3(K_v)\,,\quad\forall\;v\in\Omega_A\,.
\]By Lemma$\;$\ref{lemma2p1temp}, we have $q\cong q'$ over $K$. This
means that $\xi\in H^1(K,\,G)$ lies in the kernel of
\[
\eta\,:\;\;H^1(K\,,\,G)\lra H^1(K\,,\,\bfSO(q))\,.
\]By the exactness of the sequence \eqref{eq3p5p1temp},
$\xi=\psi(\al)$ for some $\al\in
\mathrm{Coker}(\delta)=\frac{K^*/K^{*2}}{\Sn(q_K)}$. Consider now the
following commutative diagram with exact rows
\[
\begin{CD}
1 @>>> \frac{K^*/K^{*2}}{\Sn(q_K)} @>\psi>> H^1(K\,,\,G) @>\eta>>
H^1(K\,,\,\bfSO(q))\\
&& @VVV @VVV @VVV\\
1 @>>> \prod_{v\in\Omega_A}\frac{K^*_v/K^{*2}_v}{\Sn(q_{K_v})}
@>\psi>> \prod_{v\in\Omega_A}H^1(K_v\,,\,G) @>\eta>>
\prod_{v\in\Omega_A}H^1(K_v\,,\,\bfSO(q))
\end{CD}
\]The canonical image $\al_v$ of $\al$ in $\frac{K^*_v/K^{*2}_v}{\Sn(q_{K_v})}$
is trivial for all $v\in\Omega_A$. From
Proposition$\;$\ref{prop3p2temp} and Corollary$\;$\ref{coro3p4temp},
it follows that $\al=1$ and hence $\xi=\psi(\al)$ is trivial.

For (ii), assume the Rost invariant $R_G(\xi)$ of $\xi$ is trivial.
Then the Arason invariant $e_3([q\bot (-q')])$ is zero. Since
$\mathrm{cd}_2(K)\le 3$, the map $e_3: I^3(K)\to H^3(K,\,\mathbb{Z}/2)$ is
injective. So we get $q\cong q'$ over $K$ and therefore
$\xi=\psi(\al)$ for some $\al\in \frac{K^*/K^{*2}}{\Sn(q_K)}$. When
the rank of $q$ is $\ge 5$, we have $K^*/K^{*2}=\Sn(q_K)$ by
Corollary$\;$\ref{coro3p4temp}. So $\al=1$ and $\xi$ is trivial.
\end{proof}

\begin{remark}\label{remark3p6temp}
Assertion (ii) of Theorem$\;$\ref{thm3p5temp} may be compared with
the following result, which was already known to experts (cf.
\cite[Prop.$\;$5.2]{CTPaSu}): Let $E$ be a field of characteristic
$\neq 2$ and of cohomological 2-dimension $\mathrm{cd}_2(E)\le 3$.
Let $q$ be an \emph{isotropic} quadratic form of rank $\ge 5$ over
$E$. Then the Rost invariant
\[
H^1(E\,,\,\bfSpin(q))\lra H^3(E\,,\,\mathbb{Q}/\mathbb{Z}(2))
\]for the spinor group $\bfSpin(q)$ has a trivial kernel.
\end{remark}

\section{Groups of type $D_n^*$}\label{sec5}

\begin{prop}\label{prop5p1NEW}Let $K$ be the function field of a $p$-adic arithmetic surface or a
local henselian surface with finite residue field of characteristic
$p$. Assume $p\neq 2$ in the local henselian case. Let
$(D\,,\,\sigma)$ be a quaternion division algebra with an orthogonal
involution over $K$ and let $h$ be a hermitian form of rank $\ge 2$
over $(D\,,\,\sigma)$.

Then the natural map
\[
\frac{K^*/K^{*2}}{\Sn(h_K)}\lra
\prod_{v\in\Omega_A}\frac{K^*_v/K^{*2}_v}{\Sn(h_{K_v})}
\]
is injective.
\end{prop}
\begin{proof}
First assume $\rank(h)=2$. Put $d=\disc(h)\in K^*/K^{*2}$. If
$d=1\in K^*/K^{*2}$, then $h$ is isotropic and $\Sn(h)=\Nrd(D^*)$
modulo squares by Merkurjev's norm principle
(Theorem$\;$\ref{thm4p2temp}). The result then follows from
Theorem$\;$\ref{thm1p4temp}. Let us assume $d=\disc(h)\in
K^*/K^{*2}$ is nontrivial. Let
$(A\,,\,\tilde{\sigma})=(\rM_2(D)\,,\,\sigma_h)$, where $\sigma_h$
denotes the adjoint involution of $h$ on $A=\rM_2(D)$. The even
Clifford algebra $C=C_0(A\,,\,\tilde{\sigma})$ of the pair
$(A\,,\,\tilde{\sigma})$ (cf. \cite[$\S$8]{KMRT}) is a quaternion
algebra over the field $K(\sqrt{d})$ and one has
\[
\Sn(h_K)=\Nrd(C^*)\cap K^*\;\pmod{K^{*2}}\,.
\](cf. \cite[p.94, Thm.$\;$8.10 and p.214, Coro.$\;$15.11]{KMRT}.)
As in the proof of Proposition$\;$\ref{prop3p2temp}, it follows from
Theorem$\;$\ref{thm1p4temp} that an element $\lambda\in K^*/K^{*2}$
is a spinor norm for $h_K$ if and only if it is a spinor norm for
$h_{K_v}$ for all $v\in\Omega_A$.

Assume next $\rank(h)\ge 3$. Let $\lambda\in K^*$ and assume
$\lambda$ is a local spinor norm for $h_{K_v}$ for every
$v\in\Omega_A$. Merkurjev's norm principle
(Theorem$\;$\ref{thm4p2temp}) implies that $\lambda\in
\Nrd(D^*_{K_v})$ for every $v\in\Omega_A$. Hence
$\lambda\in\Nrd(D^*)$ by Theorem$\;$\ref{thm1p4temp}. (Note that
$K^{*2}\subseteq \Nrd(D^*)$ since $D$ is a quaternion algebra.) Let
$K'/K$ be a  field extension such that $D_{K'}$ is split and
$\lambda=N_{K'/K}(\mu)$ for some $\mu\in (K')^*$. By
Corollary$\;$\ref{coro2p14NEW}, $N_{K'/K}(\Sn(h_{K'}))\subseteq
\Sn(h_K)$. Since $\lambda\in K^*/K^{*2}$ lies in the image of
$N_{K'/K}: (K')^*/(K')^{*2}\to K^*/K^{*2}$, to show $\lambda$ is a
spinor norm for $h_K$ it suffices to show that the map
\[
\delta'\,:\;\bfSU(h)(K')\lra (K')^*/(K')^{*2}
\]is surjective. Note that $D$ splits over $K'$
by the choice of $K'$. So we see from (\ref{para4p3temp}) that
$\mathrm{Im}(\delta')=\Sn(h_{K'})=\Sn(q_{K'})$, where $q_{K'}$ is a
quadratic form of rank $2.\rank(h)\ge 6$ over $K'$. Now the result
follows immediately from Corollary$\;$\ref{coro3p4temp}.
\end{proof}

\begin{prop}\label{prop4p5temp}Let $K$ be the function field of a $p$-adic arithmetic surface or a
local henselian surface with finite residue field of characteristic
$p$. Assume  $p\neq 2$ in the local henselian case. Let
$(D\,,\,\sigma)$ be a quaternion division algebra with an orthogonal
involution over $K$. Let $h$ be a nonsingular hermitian form of even
rank $\ge 2$ over $(D\,,\,\sigma)$. Assume that $h$ has trivial
discriminant, trivial Clifford invariant and trivial Rost invariant
$($cf. $\S\ref{sec2p2})$.

If the form $h_{K_v}$ over
$(D_{K_v}\,,\,\sigma)=(D\otimes_KK_v\,,\,\sigma)$ is hyperbolic for
every $v\in\Omega_A$, then the form $h$ over $(D\,,\,\sigma)$ is
hyperbolic.
\end{prop}
\begin{proof}
Let $L\subseteq D$ be a subfield which is a quadratic extension over
$K$ such that $\sigma(L)=L$ and $\sigma|_L=\mathrm{id}_L$. Such an $L$ exists since $\sigma$ is an orthogonal involution. Let $\mu\in D^*$
be an element such that $\sigma(\mu)=-\mu$, $\Int(\mu)(L)=L$ and
$\Int(\mu)|_L=\iota$, where $\iota$ denotes the nontrivial element
of the Galois group $\Gal(L/K)$. The involution
$\tau:=\Int(\mu)\circ \sigma$ is a symplectic involution on $D$ (and
hence coincides with the canonical involution on the quaternion
algebra $D$). The ``key exact sequence'' of Parimala-Sridharan-Suresh (cf. \eqref{eq2p7p2NEW}) yields the following commutative diagram with exact
rows
\[
\xymatrix{ W(D\,,\,\tau) \ar[d]_{} \ar[r]^{\pi_1} &  W(L\,,\,\iota)
\ar[d]_{} \ar[r]^{\wt{\rho}}
& W(D\,,\,\sigma) \ar[d]_{} \ar[r]^{\wt{\pi}_2} & W(L) \ar[d]_{} \\
\prod_{v\in\Omega_A}W(D_v\,,\,\tau)\ar[r]^{\pi_1} &
\prod_{v\in\Omega_A}W(L_v\,,\,\iota)  \ar[r]^{\wt{\rho}} &
\prod_{v\in\Omega_A} W(D_v\,,\,\sigma) \ar[r]^{\wt{\pi}_2} &
\prod_{v\in\Omega_A}W(L_v) }
\]
where for any $K$-algebra $B$ we denote $B_v=B\otimes_KK_v$ for each
$v\in\Omega_A$. (Here $L_v$ need not be a field. It can be a Galois $K_v$-algebra of the form $L_{w_1}\times L_{w_2}$, where $w_1,\,w_2$ are discrete valuations of $L$ lying over $v$. But this does not affect the construction of the key exact sequence for $D_v$. Indeed, the same choice of $\mu\in D^*\subseteq D_v^*$ satisfies the condition that $\mathrm{Int}(\mu)|_{L_v}$ is the nontrivial automorphism of the $K_v$-algebra $L_v$. It is not difficult to check that the key exact sequence for $D_v$ is still well defined.)

 The form $\wt{\pi}_2(h)\in W(L)$ has even rank,
trivial discriminant and trivial Clifford invariant by
\cite[Prop.$\;$3.2.2]{BP1}. Hence $\wt{\pi}_2(h)\in I^3(L)\subseteq
W(L)$. Let $\Omega_{A'}$ denote the set of divisorial valuations of
$L$. Then for every $w\in\Omega_{A'}$ one has $\wt{\pi}_2(h)=0$ in
$W(L_w)$. By Lemma$\;$\ref{lemma2p1temp}, $\wt{\pi}_2(h)=0$ in
$W(L)$. So by the exactness of the first row in the above diagram,
there exists a hermitian form of even rank $h_0$ over
$(L\,,\,\iota)$ such that $\wt{\rho}(h_0)=h\in W(D\,,\,\sigma)$.

Let $\al=\disc(h_0)\in K^*/N_{L/K}(L^*)$ be the discriminant of
$h_0$. One has
\[
\mathscr{C}\ell(\wt{\rho}(h_0))=(L\,,\,\al)\;\in\;\;{}_2\Br(K)/(D)
\]by \cite[Prop.$\;$3.2.3]{BP1}. Since $\mathscr{C}\ell(\wt{\rho}(h_0))=\mathscr{C}\ell(h)=0$
by assumption, one has either $(L,\,\al)=0$ or $(L\,,\,\al)=(D)$ in
$\Br(K)$. If $(L\,,\,\al)=0\in\Br(K)$ then $\al$ is a norm for the
extension $L/K$ so that $\disc(h_0)=1\in K^*/N_{L/K}(L^*)$. If
$(L\,,\,\al)=D$, writing $L=K(\sqrt{a})$ such that
$D=(a\,,\,\al)_K$, one has $\disc(\langle 1\,,\,-\al\rangle)=\al\in
K^*/N_{L/K}(L^*)$. By the construction of the map $\pi_1$, one has
$\pi_1(\langle 1\rangle)=\langle 1\,,\,-\al\rangle\in
W(L\,,\,\iota)$ (since $D=L\oplus \mu L$ with $\mu^2=\al$).
Replacing $h_0$ by $h_0-\pi_1(\langle 1\rangle)$, we may assume that
$\disc(h_0)=1\in K^*/N_{L/K}(L^*)$. Let $2n=\rank(h_0)$ and let
$\bfSU_{2n}(L\,,\,\iota)$ denote the special unitary group of the
hyperbolic form $\smatr{0}{I_n}{I_n}{0}$ over $(L\,,\,\iota)$. The
form $h_0$, having trivial discriminant, now determines a class in
$H^1(K\,,\,\bfSU_{2n}(L\,,\,\iota))$.

Let $H_{2n}$ be the hyperbolic form  $\smatr{0}{I_n}{I_n}{0}$ over
$(D\,,\,\sigma)$ and let $\bfU_{2n}(D\,,\,\sigma)$,
$\bfSU_{2n}(D\,,\,\sigma)$ and $\bfSpin_{2n}(D\,,\,\sigma)$ denote
respectively the unitary group, the special unitary group and the
spin group of the form $H_{2n}$. By (\ref{para2p12NEW}), there is a
homomorphism
\[
\rho_0\,:\;\;\bfSU_{2n}(L\,,\,\iota)\lra
\bfSpin_{2n}(D\,,\,\sigma)\,.
\]which induces a commutative diagram
\[
\xymatrix{ H^1(K\,,\,\bfSU_{2n}(L\,,\,\iota)) \ar[dr]_{\rho'}
\ar[rr]^{\rho_0} && H^1(K\,,\,\bfSpin_{2n}(D\,,\,\sigma)) \ar[dl]_{} \\
& H^1(K\,,\,\bfU_{2n}(D\,,\,\sigma)) & }
\]such that the map $\wt{\rho}\,:\;W(L\,,\,\iota)\to W(D\,,\,\sigma)$ in the ``key exact sequence'' \eqref{eq2p7p2NEW}
restricted to forms of
rank $2n$ and of trivial discriminant is compatible with the map
$\rho'$ at the level of cohomology sets.

By \cite[Prop.$\;$3.20]{BP2}, one has
\[
R_{\bfSpin_{2n}(D\,,\,\sigma)}(\rho_0([h_0]))=R_{\bfSU_{2n}(L\,,\,\iota)}([h_0])\,\in\;
H^3(K\,,\,\mathbb{Q}/\mathbb{Z}(2))\,.
\]Thus by the definition of the Rost invariant $\sR$ (cf. (\ref{para2p11NEW})),
\[
0=\sR(h)=[R_{\bfSpin_{2n}(D\,,\,\sigma)}(\rho_0([h_0]))]=
[R_{\bfSU_{2n}(L\,,\,\iota)}([h_0])]\,\in\;\frac{H^3(K\,,\,\mathbb{Q}/\mathbb{Z}(2))}{H^1(K\,,\,\mu_2)\cup
(D)}\,.
\]Therefore, there is an element $\beta\in K^*/K^{*2}=H^1(K\,,\,\mu_2)$ such that
\[
R_{\bfSU_{2n}(L\,,\,\iota)}([h_0])=(\beta)\cup (D)\,\in\;
H^3(K\,,\,\mathbb{Q}/\mathbb{Z}(2))\,.
\]A direct computation shows that the element $\wt{h}_0:=\pi_1(\langle 1\,,\,-\beta\rangle)\in W(L\,,\,\iota)$
has associated trace form $q_{\wt{h}_0}=\langle
1\,,\,-\beta\rangle\otimes n_D$, where $n_D$ denotes the norm form
of the quaternion algebra $D$. By \cite[p.438,
Example$\;$31.44]{KMRT}, the class of $\wt{h}_0$ has Rost invariant
\[
R_{\bfSU_{4}(L\,,\,\iota)}([\wt{h}_0])=e_3(q_{\wt{h}_0})=(\beta)\cup
(D)\;\in\;H^3(K\,,\,\mathbb{Q}/\mathbb{Z}(2))\,.
\]Modifying $h_0$ by $\wt{h}_0=\pi_1(\langle 1\,,\,-\beta\rangle)$, we may further assume that
the class $[h_0]\in H^1(K\,,\,\bfSU_{2n}(L\,,\,\iota))$ has trivial
Rost invariant, i.e., $e_3(q_{h_0})=0$. Since $\mathrm{cd}_2(K)\le
3$, the Arason invariant $e_3: I^3(K)\to H^3(K\,,\,\mathbb{Z}/2)$ is
injective. Hence $[q_{h_0}]=0\in W(K)$ and $[h_0]=0\in
W(L\,,\,\iota)$ by (\ref{para2p4NEW}) (cf. \cite[p.348,
Thm.$\;$10.1.1]{Schar}). It then follows immediately that
$[h]=\wt{\rho}([h_0])=0\in W(D\,,\,\sigma)$.
\end{proof}

\begin{coro}\label{coro4p6temp}Let $K$ be the function field of a $p$-adic arithmetic surface or a
local henselian surface with finite residue field of characteristic
$p$. Assume $p\neq 2$ in the local henselian case. Let $(D\,,\,\sigma)$
be a quaternion division algebra with an orthogonal involution over
$K$. Let $h_1,\,h_2$ be hermitian forms over $(D\,,\,\sigma)$ with
the same rank and discriminant such that
\[
\mathscr{C}\ell(h_1\bot (-h_2))=0\,\in\;{}_2\Br(K)/(D)\,
\]and
\[
\sR(h_1\bot (-h_2))=0\,\in\;
H^3(K\,,\,\mathbb{Q}/\mathbb{Z}(2))/H^1(K\,,\,\mu_2)\cup (D)\,.
\]Then $h_1\cong h_2$ if and only if $(h_1)_{K_v}\cong (h_2)_{K_v}$ for every $v\in\Omega_A$.
\end{coro}
\begin{proof}
Apply Proposition$\;$\ref{prop4p5temp} to the form $h=h_1\bot
(-h_2)$ and use Witt's cancellation theorem.
\end{proof}

\begin{thm}\label{thm4p7temp}Let $K$ be the function field of a $p$-adic arithmetic surface or a
local henselian surface with finite residue field of characteristic
$p$. Assume $p\neq 2$ in the local henselian case. Let
$(D,\,\sigma)$ be a quaternion division algebra with an orthogonal
involution over $K$, $h$ a nonsingular hermitian form of rank $\ge
2$ over $(D\,,\,\sigma)$ and $G=\bfSpin(h)$.

Then the natural map
\[
H^1(K\,,\,G)\lra \prod_{v\in\Omega_A}H^1(K_v\,,\,G)
\]has a trivial kernel.
\end{thm}
\begin{proof}
Let $\xi\in H^1(K\,,\,\bfSpin(h))$ be a class which is
trivial in $H^1(K_v\,,\,\bfSpin(h))$ for all $v\in \Omega_A$. The
image of $\xi$ under the composite map
\[
H^1(K\,,\,G)=H^1(K\,,\,\bfSpin(h))\lra H^1(K\,,\,\bfSU(h))\lra
H^1(K\,,\,\bfU(h))
\]is the class of a hermitian form $h'$ which has the same rank and discriminant as $h$ such that
\[
\mathscr{C}\ell(h\bot (-h'))=0\,\in\;{}_2\Br(K)/(D)\,.
\]Let $n=\rank(h)$. Let $\bfSpin_{2n}(D\,,\,\sigma)$ and $\bfU_{2n}(D\,,\,\sigma)$ denote respectively the
spin group and the unitary group of the hyperbolic form
$\smatr{0}{I_n}{I_n}{0}$ over $(D\,,\,\sigma)$. Then the class
$[h\bot (-h')]\in H^1(K\,,\,\bfU_{2n}(D\,,\,\sigma))$ lifts to an
element $\xi'\in H^1(K\,,\,\bfSpin_{2n}(D\,,\,\sigma))$. By
\cite[Lemma$\;$5.1]{PaPr}, we have
\begin{equation}\label{eq4p7p1temp}
[R_G(\xi)]=\sR(h\bot(-h'))=[R_{\bfSpin_{2n}(D\,,\,\sigma)}(\xi')]\,\in\;\frac{H^3(K\,,\,\mathbb{Q}/\mathbb{Z}(2))}{H^1(K\,,\,\mu_2)\cup
(D)}\,.
\end{equation}Since $\xi$ is locally trivial, the commutative diagram
\[
\begin{CD}
H^1(K\,,\,G) @>R_G>> H^3(K\,,\,\mathbb{Q}/\mathbb{Z}(2)) \\
@VVV @VVV \\
\prod_{v\in\Omega_A}H^1(K_v\,,\,G) @>R_G>>
\prod_{v\in\Omega_A}H^3(K_v\,,\,\mathbb{Q}/\mathbb{Z}(2))
\end{CD}
\]shows that the Rost invariant $R_G(\xi)$ is locally trivial. By Theorem$\;$\ref{thm1p3temp}, noticing that the Rost invariant
$R_G$ takes values in the subgroup $H^3(K\,,\,\mu_4^{\otimes 2})$,
we  get $R_G(\xi)=0\in H^3(K\,,\,\mathbb{Q}/\mathbb{Z}(2))$. Thus, by
\eqref{eq4p7p1temp},
\[
\sR(h\bot
(h'))=0\,\in\;\frac{H^3(K\,,\mathbb{Q}/\mathbb{Z}(2))}{H^1(K\,,\,\mu_2)\cup (D)}\,.
\]Now Corollary$\;$\ref{coro4p6temp} implies that $h\cong h'$ and hence the image of $\xi\in H^1(K\,,\,G)$ in $H^1(K,\,\bfU(h))$ is trivial.
By \cite[Lemma$\;$7.11]{BP2}, the canonical image of $\xi$ in
$H^1(K,\,\bfSU(h))$ is also trivial.

Now consider the following commutative diagram with exact rows
\[
\begin{CD}
1 @>>> \frac{K^*/K^{*2}}{\Sn(h_K)} @>\varphi>> H^1(K\,,\,G) @>>> H^1(K\,,\,\bfSU(h))\\
&& @VVV @VVV @VVV \\
1 @>>> \prod_{v\in\Omega_A}\frac{K_v^*/K_v^{*2}}{\Sn(h_{K_v})} @>>>
\prod_{v\in\Omega_A}H^1(K_v\,,\,G) @>>> \prod_{v\in\Omega_A}
H^1(K_v\,,\,\bfSU(h))
\end{CD}
\]which is induced by the natural exact sequence of algebraic groups
\[
1\lra \mu_2\lra G=\bfSpin(h)\lra \bfSU(h)\lra 1\,.
\]The exactness of the first row yields $\xi=\varphi(\theta)$ for some $\theta\in \frac{K^*/K^{*2}}{\Sn(h_K)}$. The commutative
diagram then shows that $\theta$ is locally trivial since $\xi$ is
locally trivial. From Propositoin$\;$\ref{prop5p1NEW} it follows
that $\theta=1\in \frac{K^*/K^{*2}}{\Sn(h_K)}$ and hence
$\xi=\varphi(\theta)$ is trivial in $H^1(K\,,\,G)$. This completes the
proof.
\end{proof}

\section{Groups of type ${}^2A_n^*$}\label{sec6}

\subsection{Case of odd index}

\begin{prop}\label{prop5p1temp}Let $K$ be the function field of a $p$-adic arithmetic surface or a
local henselian surface with finite residue field of characteristic
$p$. Assume $p\neq 2$ in the local henselian case. Let $L/K$ be a
quadratic field extension, $(D\,,\,\tau)$ a central division algebra
of odd degree over $L$ with an $L/K$-involution $\tau\,($i.e., a
unitary involution $\tau$ such that $L^{\tau}=K\,)$. Let $h_1,\,h_2$
be nonsingular hermitian forms over $(D\,,\,\tau)$ which have the
same rank and discriminant.

If the forms $(h_1)_{K_v}\cong (h_2)_{K_v}$ over
$(D_{K_v}\,,\,\tau)=(D\otimes_LL\otimes_KK_v\,,\,\tau)$ are
isomorphic
 for all $v\in\Omega_A$, then the forms $h_1\,,\,h_2$ over $(D\,,\,\tau)$
are isomorphic.
\end{prop}
\begin{proof}
Let $M/K$ be a field extension of odd degree such that
$D_M=D\otimes_L(L\otimes_KM)$ is split over the field
$LM=L\otimes_KM$. (Such an extension $M/K$ exists by
\cite[Lemma$\;$3.3.1]{BP1}.) The base extension $\tau_M$ of $\tau$
is a unitary involution on the central simple $(LM)$-algebra $D_M$
such that $(LM)^{\tau_M}=M$. Let $\iota$ denote the nontrivial
element of the Galois group $\Gal(L/K)$ and regard
$\iota_M\in\Gal(LM/M)$ as a unitary involution on $LM$. There is a
nonsingular hermitian form $(V\,,\,f)$ over $(LM\,,\,\iota_M)$ such
that $(D_M\,,\,\tau_M)\cong (\End_{LM}(V)\,,\,\iota_f)$, where
$\iota_f$ denotes the adjoint involution on $\End_{LM}(V)$ with
respect to $f$ (cf. \cite[p.43, Thm.$\;$4.2 (2)]{KMRT}). We have a
Morita equivalence between the category of hermitian forms over
$(D_M\,,\,\tau_M)$ and the category of hermitian forms over
$(LM\,,\,\iota_M)$ (cf. (\ref{para2p6NEW})), which induces an
isomorphism of Witt groups
\[
\phi_f\,:\;\;W(D_M\,,\,\tau_M)\simto W(LM\,,\,\iota_M)\,.
\]

Let $h=h_1\bot (-h_2)$ and let $h_M$ be its base extension over
$(D_M\,,\,\tau_M)$. Via the Morita equivalence mentioned above,
$h_M$ corresponds to a hermitian form $\wt{h}_M$ over
$(LM\,,\,\iota_M)$. Let $q_M:=q_{\wt{h}_M}$ be the trace form of
$\wt{h}_M$ (which is a quadratic form over the field $M$). Since $h$
has even rank and trivial discriminant, the class $[q_M]\in W(M)$ of
the quadratic form $q_M$ lies in $I^3(M)$. The hypothesis on the
local triviality (with respect to $\Omega_A$) of $[h]=[h_1\bot
(-h_2)]$ implies that $[q_M]\in I^3(M)$ is locally trivial (with
respect to the set of discrete valuations of $M$ defined in the same
way as $\Omega_A$). By Lemma$\;$\ref{lemma2p1temp}, we have
$[q_M]=0$ and hence $[\wt{h}_M]=0$ in $W(LM\,,\,\iota_M)$. Since
$W(D_M\,,\,\tau_M)\cong W(LM\,,\,M)$, $[h_M]=0$ in
$W(D_M\,,\,\tau_M)$. Since $M/K$ is an odd degree extension, the
natural map $W(D\,,\,\tau)\to W(D_M\,,\,\tau_M)$ is injective by a
theorem of Bayer-Fluckiger and Lenstra (cf. \cite[p.80,
Coro.$\;$6.18]{KMRT}). So we get $[h]=0$ in $W(D\,,\,\tau)$, thus
proving the proposition.
\end{proof}

\begin{lemma}\label{lemma5p2temp}Let $K$ be the function field of a $p$-adic arithmetic surface or a
local henselian surface with finite residue field of characteristic
$p$. Let $L/K$ be a separable quadratic field extension
and $(D\,,\,\tau)$ a central division $L$-algebra of square-free index
 $\ind(D)$ with a unitary involution $\tau$ such that $L^{\tau}=K$.
Assume $p\nmid \ind(D)$ in the local henselian case.

Then for any nonsingular
hermitian form $h$ over $(D\,,\,\tau)$, the natural map
\[
\frac{(R^1_{L/K}\mathbb{G}_m)(K)}{\Nrd(\bfU(h)(K))}\lra
\prod_{v\in\Omega_A}\frac{(R^1_{L/K}\mathbb{G}_m)(K_v)}{\Nrd(\bfU(h)(K_v))}
\]is injective.
\end{lemma}
\begin{proof}
First assume $\ind(D)=2$ so that $D$ is a quaternion division
algebra over $L$. By \cite[p.202, Exercise$\;$III.12 (a)]{KMRT}, we
have
\[
\Nrd(\bfU(h)(K))=\set{z\tau(z)^{-1}\,|\,z\in\Nrd(D^*)}=\Nrd(\bfU_2(D\,,\,\tau)(K))\,,
\]where $\bfU_2(D\,,\,\tau)$ denotes the unitary group of the rank 2 hyperbolic form $\smatr{0}{1}{1}{0}$ over
$(D\,,\,\tau)$. So we may assume that $h=\smatr{0}{1}{1}{0}$. The
exact sequence of algebraic groups
\[
1\lra \bfSU_2(D\,,\,\tau)\lra
\bfU_2(D\,,\,\tau)\overset{\Nrd}{\lra}(R^1_{L/K}\mathbb{G}_m)\lra 1
\]gives rise to the following commutative diagram with exact rows
\[
\begin{CD}
1 @>>> \frac{(R^1_{L/K}\mathbb{G}_m)(K)}{\Nrd(\bfU(h)(K))}
 @>\varphi>> H^1(K\,,\,\bfSU_2(D\,,\,\tau))\\
&& @VVV @VVV  \\
1 @>>>
\prod_{v\in\Omega_A}\frac{(R^1_{L/K}\mathbb{G}_m)(K_v)}{\Nrd(\bfU(h)(K_v))}
 @>>> \prod_{v\in\Omega_A}H^1(K_v\,,\,\bfSU_2(D\,,\,\tau))
\end{CD}
\]We need only to show that the vertical map on the right in the above diagram is injective.

By \cite[p.26, Prop.$\;$2.22]{KMRT}, there is a unique quaternion
$K$-algebra $D_0$ contained in $D$ such that $D=D_0\otimes_KL$ and
$\tau=\tau_0\otimes\iota$, where $\tau_0$ is the canonical
involution on $D_0$ and $\iota$ is the nontrivial element in the
Galois group $\Gal(L/M)$. Write $L=K(\sqrt{d})$ and let $n_{D_0}$ be
the norm form of the quaternion $K$-algebra $D_0$. Then by
\cite[p.229]{KMRT}, we have $\bfSU_2(D\,,\,\tau)=\bfSpin(q)$, where
$q=\langle 1\,,\,-d\rangle\otimes n_{D_0}$. Now the result follows
from Theorem$\;$\ref{thm3p5temp}.

Assume next $\ind(D)$ is odd (and square-free). By \cite[p.202,
Exercise$\;$III.12 (b)]{KMRT},
\[
\Nrd(\bfU(h)(K))=\Nrd(D^*)\cap (R^1_{L/K}\mathbb{G}_m)(K)\,.
\]Let $\lambda\in (R^1_{L/K}\mathbb{G}_m)(K)=\set{z\in L^*\,|\, N_{L/K}(z)=1}$ be such that for every $v\in\Omega_A$,
$\lambda\in \Nrd(\bfU(h)(K_v))=\Nrd((D\otimes_KK_v)^*)\cap
(R^1_{L/K}\mathbb{G}_m)(K_v)$. Since $\ind(D)$ is square-free, it follows
from Theorem$\;$\ref{thm1p4temp} that $\lambda\in\Nrd(D^*)$. Hence
\[
\lambda\in\Nrd(\bfU(h)(K))=\Nrd(D^*)\cap (R^1_{L/K}\mathbb{G}_m)(K)\,.\]

Now assume $\mathrm{ind}(D)$ is even such that $\mathrm{ind}(D)/2$ is odd and square-free. In this case we have $D=H\otimes_LD'$ for some quaternion division algebra $H$
over $L$ and some central division algebra $D'$ of odd index over $L$. By \cite[Lemma$\;$3.3.1]{BP1}, there is an odd degree separable extension $K'/K$ such that $D'\otimes_KK'=D'\otimes_LLK'$ is split. By Morita theory, there is a unitary $LK'/K'$-involution $\sigma$ on $H\otimes_LLK'$ and a hermitian form $f$
over $(H\otimes_LLK'\,,\,\sigma)$ such that the involution $\tau$ on $D\otimes_LLK'$ is adjoint to $f$, and moreover, the form $h_{K'}$ over $(D\otimes_LLK'\,,\,\tau)$ corresponds to a hermitian form $h'$ over $(H\otimes_LLK'\,,\,\sigma)$. Consider the commutative diagram
\[
\begin{CD}
\frac{R^1_{L/K}\mathbb{G}_m(K)}{\Nrd(\bfU(h)(K))} @>\eta>> \prod_{v\in\Omega_A}\frac{R^1_{L/K}\mathbb{G}_m(K_v)}{\Nrd(\bfU(h)(K_v))}\\
@VVV @VVV\\
\frac{R^1_{LK'/K'}\mathbb{G}_m(K')}{\Nrd(\bfU(h')(K'))} @>\eta'>> \prod_{v\in\Omega_A}\frac{R^1_{LK'/K'}\mathbb{G}_m(K'_v)}{\Nrd(\bfU(h')(K'_v))}\\
\end{CD}
\]The map $\eta'$ is already shown to be injective. Let $\lambda\in R^1_{L/K}\mathbb{G}_m(K)\subseteq L^*$ be an element which is a reduced norm for
$\bfU(h)(K_v)$ for every $v$. Then, considered as an element of $R^1_{LK'/K'}(K')\subseteq (LK')^*$, $\lambda$ lies in $\Nrd(\bfU(h')(K'))$.
By \cite[Prop.$\;$10.2]{PaPr}, we have
\[
N_{LK'/K'}(\Nrd(\bfU(h')(K')))\subseteq \Nrd(\bfU(h)(K))\,.
\]Hence, $\lambda^{2r+1}\in \Nrd(\bfU(h)(K))$, where $2r+1=[K': K]$. It is sufficient to show that $\lambda^2\in \Nrd(\bfU(h)(K))$. For this, we choose a quadratic extension $M/K$ such that $H\otimes_KM=H\otimes_LLM$ is split. A similar argument as above, using the result in the case of odd index this time, shows
that $\lambda\in \Nrd(\bfU(h_M)(M))$. Thus,
\[
\lambda^2=N_{LM/M}(\lambda)\in
N_{LM/M}(\Nrd(\bfU(h_M)(M)))\subseteq \Nrd(\bfU(h)(K))\,.
\]This completes the proof of the lemma.
\end{proof}

\begin{thm}\label{thm5p3temp}Let $K$ be the function field of a $p$-adic arithmetic surface or a
local henselian surface with finite residue field of characteristic
$p$. Let  $L/K$ be a separable quadratic field extension
and $(D\,,\,\tau)$ a central division $L$-algebra with a
unitary $L/K$-involution whose index $\ind(D)$ is
odd and square-free. Assume further that $p\nmid 2.\ind(D)$ in the local henselian case.

Then for any nonsingular hermitian form $h$ over $(D\,,\,\tau)$, the
natural map
\[
H^1(K\,,\,\bfSU(h))\lra \prod_{v\in\Omega_A}H^1(K_v\,,\,\bfSU(h))
\]has a trivial kernel.
\end{thm}
\begin{proof}
Let $\xi\in H^1(K\,,\,\bfSU(h))$ be a class that is locally trivial
in $H^1(K_v\,,\,\bfSU(h))$ for every $v\in\Omega_A$. Let $h'$ be a
hermitian form whose class $[h']\in H^1(K\,,\,\bfU(h))$ is the image
of $\xi$ under the natural map $H^1(K\,,\,\bfSU(h))\to
H^1(K\,,\,\bfU(h))$. The two forms $h'$ and $h$ have the same rank
and discriminant, and they are locally isomorphic since $\xi$ is
locally trivial. So by Proposition$\;$\ref{prop5p1temp}, $h'\cong h$
as hermitian forms over $(D\,,\,\tau)$. This means that $\xi\in
H^1(K\,,\,\bfSU(h))$ maps to the trivial element in
$H^1(K\,,\,\bfU(h))$.

Consider now the following commutative diagram with exact rows
\[
\begin{CD}
1 @>>> \frac{(R^1_{L/K}\mathbb{G}_m)(K)}{\Nrd(\bfU(h)(K))}
 @>\varphi>> H^1(K\,,\,\bfSU(h)) @>>> H^1(K\,,\,\bfU(h)) \\
&& @V{\eta}VV @VVV @VVV  \\
1 @>>>
\prod_{v\in\Omega_A}\frac{(R^1_{L/K}\mathbb{G}_m)(K_v)}{\Nrd(\bfU(h)(K_v))}
 @>>> \prod_{v\in\Omega_A}H^1(K_v\,,\,\bfSU(h)) @>>> \prod_{v\in\Omega_A}H^1(K_v\,,\,\bfU(h))
\end{CD}
\]There is an element $\theta\in (R^1_{L/K}\mathbb{G}_m)(K)/\Nrd(\bfU(h)(K))$ such that $\varphi(\theta)=\xi$.
The map $\eta$ is injective by Lemma$\;$\ref{lemma5p2temp}. So we
have $\theta=1$ and $\xi=\varphi(\theta)$ is trivial. The theorem is
thus proved.
\end{proof}

\subsection{Some observations on Suresh's exact sequence}

\newpara\label{para6p4TEMP}
Let $E$ be a field of characteristic $\neq 2$.
Let $D$ be a quaternion division algebra over a quadratic field extension $L$ of $E$. Let
$\tau$ be a unitary $L/E$-involution on $D$.
There is a unique quaternion $E$-algebra $D_0$ contained in $D$ such that
$D=D_0\otimes_EL$ and $\tau=\tau_0\otimes\iota$, where $\tau_0$ is the canonical (symplectic) involution on
$D_0$ and $\iota$ is the nontrivial element of the Galois group $\Gal(L/E)$. Then we have Suresh's exact sequence
(cf. (\ref{para2p8v2}))
\[
W(L)\overset{\wt{\pi}_1}{\lra}W(D_0\,,\,\tau_0)\overset{\wt{\rho}}{\lra}W(D\,,\,\tau)\overset{p_2}{\lra}W^{-1}(D_0,\,\tau_0)\,.
\]The goal of this subsection is to analyze the image of the map $\wt{\pi}_1$ in this sequence.

\

\newpara\label{para1p2TEMP}
With notation as in (\ref{para6p4TEMP}), let $h_0$ be a hermitian form of rank $m$ over $(D_0,\,\tau_0)$. Let $M(h_0)\in A:=\rM_m(D_0)$ be a representation matrix
of $h_0$. One can define the pfaffian norm $\mathrm{Pf}(h_0)$ as the pfaffian norm of $M(h_0)\in A$ with respect to the adjoint involution of $h_0$ on $A$
(cf. \cite[p.19]{KMRT}). This is a well defined element of the group $E^*/\Nrd(D_0^*)$. If $h_0=\langle \al_1,\dotsc, \al_m\rangle$ with
$\al_i\in E^*$, then $\mathrm{Pf}(h_0)$ is represented by the discriminant of the quadratic form $\langle \al_1,\dotsc, \al_m\rangle$ over $E$.

\begin{lemma}\label{lemma1TEMP} With notation as in $(\ref{para6p4TEMP})$, write $L=E(\sqrt{d})$ with $d\in E^*$.
Let $h_0$ be a hermitian form of even rank  over $(D_0,\,\tau_0)$.

$(\oi)$ If the class $[h_0]\in W(D_0,\,\tau_0)$ lies in the image of $\wt{\pi}_1$, then
its pfaffian norm $\mathrm{Pf}(h_0)\in E^*/\Nrd(D_0^*)$ lies in the subgroup generated by $N_{L/E}(L^*)$.

$(\ii)$ The converse of $(\oi)$ is true if $h_0$ is of rank $2$.
\end{lemma}
\begin{proof}
(i) For $a+b\sqrt{d}\in L^*$ with $a\,,\,b\in E$, the form $\wt{\pi}_1(\langle a+b\sqrt{d}\rangle)$ is represented by the matrix
\[
\matr{a}{bd}{bd}{ad}\,.
\]One can then verify that
\[
\wt{\pi}_1(\langle\,a+b\sqrt{d}\,\rangle)=\begin{cases}
\langle\,a\,,\,ad(a^2-b^2d)\,\rangle\;\;&\;\text{ if }\, a\neq 0\\
\langle\,2bd\,,\,-2bd\rangle\;\;&\;\text{ if }\, a=0\neq b
\end{cases}\,.
\]So it follows easily that $\mathrm{Pf}(\wt{\pi}_1(\langle\,a+b\sqrt{d}\,\rangle))$ is represented by an element of $N_{L/E}(L^*)$.

(ii) Conversely, let $h_0$ be a hermitian form of rank $2$ whose pfaffian norm $\mathrm{Pf}(h_0)$
is represented by an element of $N_{L/E}(L^*)$. We want to show $[h_0]\in\mathrm{Im}(\wt{\pi}_1)$. By Suresh's exact sequence,
it suffices to show that the form $\wt{\rho}(h_0)$ is hyperbolic over $(D\,,\,\tau)$.

We may assume $h_0=\langle\al\,,\,-\gamma\al\rangle$ with $\al,\,\gamma\in E^*$. The assumption on the pfaffian norm implies that
\[
\tau_0(u)u\gamma=
\Nrd_{D_0}(u)\gamma=a^2-b^2d\,
\] for some $u\in D_0^*$ and some $a,\,b\in E$. Since
\[
\langle\al\,,\,-\gamma\al\rangle\cong \langle \al\,,\,-\gamma\al \tau_0(u)u\rangle\quad\text{ over }\;\; (D_0\,,\,\tau_0)\,,
\]replacing $\gamma$ by $\gamma\tau_0(u)u=\gamma.\Nrd_{D_0}(u)$ if necessary, we may assume $\gamma=a^2-b^2d$ for some $a,\,b\in E$.
From the definition of the map $\wt{\rho}$, it follows easily that the form $\wt{\rho}(h_0)$ over $(D\,,\,\tau)$ is also represented by the diagonal
matrix $\langle\al\,,\,-\gamma\al\rangle$. But then for $v=(a+b\sqrt{d}\,,\,1)\in D^2$, one has
\[
\wt{\rho}(h_0)(v,\,v)=(\tau(a+b\sqrt{d})\,,\,\tau(1))\matr{\al}{0}{0}{-\gamma\al}\binom{a+b\sqrt{d}}{1}=\al(a^2-b^2d-\gamma)=0\,.
\]This show that the rank 2 form $\wt{\rho}(h_0)$ is isotropic and hence hyperbolic.
\end{proof}

\begin{lemma}\label{lemma2TEMP}
With notation as above, assume that the field $E$ has finite $u$-invariant $u(E)=r$. Then for any hermitian form $h_0$ of rank $m>r/3$ over $(D_0,\,\tau_0)$, the form $\wt{\rho}(h_0)$ over
$(D,\,\tau)$ is isotropic.
\end{lemma}
\begin{proof}
We may assume $D_0^m$ is the underlying space of the form $h_0$ and $h_0=\langle \al_1,\dotsc, \al_m\rangle$ with $\al_i\in E^*$. Then the
underlying space of $\wt{\rho}(h_0)$ is $D^m=D_0^m\oplus D^m_0\sqrt{d}$. We fix a quaternion basis $\set{1,\,i,\,j,\,ij}$ for the quaternion algebra
$D_0$. The subspace $\mathrm{Sym}(D,\,\tau)\subseteq D$ consisting of $\tau$-invariant elements is a 4-dimensional $E$-vector space with basis
\[
1,\,i\sqrt{d}\,,\,j\sqrt{d}\,,\,ij\sqrt{d}\,.\] Let $V\subseteq \mathrm{Sym}(D,\,\tau)$ be the subspace generated by $i\sqrt{d},\,j\sqrt{d}$ and
$ij\sqrt{d}$. For $w=x_1.i\sqrt{d}+x_2.j\sqrt{d}+x_3.ij\sqrt{d}$ with $x_i\in E$, a straightforward calculation yields
\[
w^2=di^2.x_1^2+dj^2.x_2^2+d(ij)^2.x_3^2\;\in\; E\,.
\]So the map
\[
\phi\,:\; V^m\lra E\;;\quad v=(v_1,\dotsc, v_m)\longmapsto \wt{\rho}(h_0)(v,\,v)=\sum \al_iv_i^2
\]defines a quadratic form of rank $3m$ over $E$. By the assumption on the $u$-invariant of $E$, the quadratic form $\phi$
is isotropic and hence the hermitian form $\wt{\rho}(h_0)$ is isotropic.
\end{proof}

\begin{lemma}\label{lemma3TEMP}
Assume that $u(E)<12$. Then for any hermitian form $h_0$ of even rank $2n$ over $(D_0,\,\tau_0)$, one has
\[
[h_0]\,\in \;\mathrm{Im}(\wt{\pi}_1)\iff \mathrm{Pf}(h_0)\,\in\;N_{L/E}(L^*).\Nrd(D_0^*)\,.
\]
\end{lemma}
\begin{proof}
In view of Lemma$\;$\ref{lemma1TEMP}, we need only to prove that if $\mathrm{Pf}(h_0)\in N_{L/E}(L^*)\Nrd(D_0^*)$, then  $[h_0]\in\mathrm{Im}(\wt{\pi}_1)$.

To prove this, we use induction on $n=\mathrm{rank}(h_0)/2$, the case $n=1$ being treated in Lemma$\;$\ref{lemma1TEMP}. Now we assume
$\mathrm{rank}(h_0)=2n\ge 4$ and $h_0$ is anisotropic. Let $V_0$ be the underlying space of $h_0$. Then the underlying space of the form $\wt{\rho}(h_0)$ is $V=V_0\oplus V_0\sqrt{d}$. By Lemma$\;$\ref{lemma2TEMP}, the form
$\wt{\rho}(h_0)$ is isotropic, that is, there is a nonzero vector $x_1+y_1\sqrt{d}\in V=V_0\oplus V_0\sqrt{d}$ such that
\[\begin{split}
0&=\wt{\rho}(h_0)(x_1+y_1\sqrt{d}\,,\,x_1+y_1\sqrt{d})\\
&=(h_0(x_1,\,x_1)-h_0(y_1,\,y_1)d)+(h_0(x_1,\,y_1)-h_0(y_1,\,x_1))\sqrt{d}\,.
\end{split}
\]Thus
\begin{equation}\label{eq5p19p1TEMP}
h_0(x_1,\,x_1)=d.h_0(y_1,\,y_1)\quad\text{and}\;\quad h_0(x_1\,,\,y_1)=h_0(y_1,\,x_1)\,.
\end{equation}Since $h_0$ is anisotropic, $h_0(x_1,\,x_1)$ and $h_0(y_1,\,y_1)$ are both nonzero and hence
lie in
\[
E^*=\set{x\in D_0^*\,|\,\tau_0(x)=x}\,.
\] In particular, $x_1\neq 0$, $y_1\neq 0$ and
\[
h_0(x_1,\,y_1)=h_0(y_1,\,x_1)\in E=\set{x\in D_0\,|\,\tau_0(x)=x}\,.
\]If $x_1=y_1\lambda$ for some $\lambda\in D_0^*$, then \eqref{eq5p19p1TEMP} yields
\[
\tau_0(\lambda)\lambda=d\quad\text{ and }\quad \tau_0(\lambda)=\lambda
\]whence $d=\lambda^2\in E^{*2}$. Since $d$ is not a square in $E$, the two vectors $x_1,\,y_1\in V_0$
generate a $D_0$-submodule $W_0:=x_1D_0+y_1D_0\subseteq V_0$ of rank 2. Put $a=h_0(y_1,\,y_1)\in E^*$ and
$bd=h_0(x_1,\,y_1)=h_0(y_1,\,x_1)\in E$. Then the restriction $f_0$ of $h_0$ to $W_0$ is represented by the matrix
\[
\matr{ad}{bd}{bd}{a}=\matr{0}{1}{1}{0}\matr{a}{bd}{bd}{ad}\matr{0}{1}{1}{0}\,.
\]A direct computation then gives
\[
\wt{\pi}_1(\langle a+b\sqrt{d}\rangle)=[f_0]\,\in\;W(D_0\,,\,\tau_0)\,.
\]This means that $h_0$ contains a subform $f_0$ of rank 2, which lies in the image of $\wt{\pi}_1$. Writing $h_0=f_0\bot g_0$, we get
$\mathrm{Pf}(g_0)\in N_{L/E}(L^*)\Nrd(D_0^*)$ since $\mathrm{Pf}(f_0)$ and $\mathrm{Pf}(h_0)$ lie in $N_{L/E}(L^*)\Nrd(D_0^*)$. Now
the induction hypothesis yields $[g_0]\in\mathrm{Im}(\wt{\pi}_1)$, whence $[h_0]=[f_0]+[g_0]\in\mathrm{Im}(\wt{\pi}_1)$.
\end{proof}

\subsection{A Hasse principle for $H^4$ of function fields of conics}

\begin{lemma}\label{lemma2p1TEMP}
Let $F$ be a field of characteristic $\neq 2$, $\overline{F}$ a separable closure of $F$ and $C\subseteq\mathbb{P}^2_F$  a smooth projective conic over $F$.
Put $\overline{C}=C\times_F\overline{F}$ and let $F(C),\, \overline{F}(C)$ denote the function fields of $C$ and $\overline{C}$ respectively.

Then the natural exact sequence
\[
0\lra \overline{F}(C)^*\otimes\mathbb{Q}_2/\mathbb{Z}_2(2)\lra \mathrm{Div}(\overline{C})\otimes\mathbb{Q}_2/\mathbb{Z}_2(2)\lra \mathrm{Pic}(\overline{C})\otimes\mathbb{Q}_2/\mathbb{Z}_2(2)\lra 0
\]induces an injection
\[
H^3(F\,,\,\overline{F}(C)\otimes\mathbb{Q}_2/\mathbb{Z}_2(2))\lra H^3(F\,,\,\mathrm{Div}(\overline{C})\otimes\mathbb{Q}_2/\mathbb{Z}_2(2))\,.
\]
\end{lemma}
\begin{proof}
Let  $C^{(1)}$ be the set of closed points of $C$. For each $P\in C^{(1)}$, let $G_P$ be the absolute Galois group of the residue field $F(P)$
of $P$. This is an open subgroup of $G=\mathrm{Gal}(\overline{F}/F)$. Write $M_P=\mathrm{Hom}_{G_P}(\mathbb{Z}[G]\,,\,\mathbb{Z})$. We have an isomorphism of abelian groups
$M_P\cong \bigoplus_{Q\mapsto P}\mathbb{Z}$, where the notation $Q\mapsto P$ means that $Q$ runs over the closed points of $\overline{C}$ lying over $P$. On the other hand, we have an isomorphism of $G$-modules:
\[
\mathrm{Div}(\overline{C})\cong\bigoplus_{P\in C^{(1)}}M_P\,.
\]
Since $C$ is a smooth projective conic, $\mathrm{Pic}(\overline{C})\cong \mathbb{Z}$ as $G$-modules. 
The natural map $\mathrm{Div}(\overline{C})\to \mathrm{Pic}(\overline{C})$ can be identified with the summation map 
\[
\sigma\,:\; \bigoplus_{P\in C^{(1)}}\bigoplus_{Q\mapsto P}\mathbb{Z}\longrightarrow \mathbb{Z}
\]So the exact sequence in the lemma may be  identified with the following
\[
0\lra \overline{F}(C)^*\otimes\mathbb{Q}_2/\mathbb{Z}_2(2)\lra \bigoplus_{P\in C^{(1)}}M_P\otimes\mathbb{Q}_2/\mathbb{Z}_2(2)\overset{\sigma}{\lra}\mathbb{Q}_2/\mathbb{Z}_2(2)\lra 0\,.
\]For any $i\ge 0$, 
\[
H^i\left(F,\,M_P\otimes\mathbb{Q}_2/\mathbb{Z}_2(2)\right)=H^i(F(P)\,,\,\mathbb{Q}_2/\mathbb{Z}_2(2))\,.
\]
It is thus sufficient to prove that the map
\[
\bigoplus_{P\in C^{(1)}} H^2(F(P)\,,\,\mathbb{Q}_2/\mathbb{Z}_2(2))\lra H^2(F\,,\,\mathbb{Q}_2/\mathbb{Z}_2(2))\,
\]
is surjective. In fact, we can choose a closed point $P\in C^{(1)}$ of degree 2 and consider the corresponding map
\[
\psi\,: \; H^2(F(P)\,,\,\mathbb{Q}_2/\mathbb{Z}_2(2))\lra H^2(F\,,\,\mathbb{Q}_2/\mathbb{Z}_2(2))\,,
\]which coincides with the corestriction map. We claim that this map is already surjective. To see this, consider for each $n\in\mathbb{N}$ the corestriction map
\[
\psi_n\,:\; H^2(F(P)\,,\,\mathbb{Z}/2^n(2))\lra H^2(F\,,\,\mathbb{Z}/2^n(2))\,.
\]By the Merkurjev--Suslin theorem, the map $\psi_n$ may be identified with the norm map
\[
N_{F(P)/F}\,:\; K_2(F(P))/2^n\lra K_2(F)/2^n
\]in Milnor's $K$-theory. The cokernel of this norm map is killed by $2=[F(P): F]$. So taking limits yields the surjectivity of the map
$\psi$. This proves the lemma.
\end{proof}

\begin{thm}\label{thm2p2TEMP}
Let $K$ be the function field of a $p$-adic arithmetic surface or a local henselian surface with finite
residue field of characteristic $p$. Assume $p\neq 2$ in the local henselian case. Let $C$ be a smooth projective conic in $\mathbb{P}^2_K$.

Then the natural map
\[
H^4(K(C)\,,\,\mathbb{Z}/2)\lra \prod_{v\in\Omega_A} H^4(K_v(C)\,,\,\mathbb{Z}/2)
\]is injective, where $v$ runs over all divisorial valuations of $K$.
\end{thm}
\begin{proof}
By the Merkurjev--Suslin theorem, we may replace $\mathbb{Z}/2$ by $\mathbb{Q}_2/\mathbb{Z}_2(3)$. Also, we may replace the completion $K_v$ by the henselisation $K_{(v)}$ for each $v$ (cf. \cite[Thm.$\;$2.9 and its proof]{Jannsen09}).
Let $\overline{K}$ be a separable closure of $K$.
Then we have a diagram of field extensions
\[
\xymatrix{
\overline{K}(C) & \overline{K} \ar[l]_{}\\
K_{(v)}(C) \ar[u]^{} & K_{(v)} \ar[l]_{} \ar[u]_{} \\
K(C) \ar[u]_{} & K \ar[l]_{} \ar[u]_{}
}
\]which identifies the Galois groups
\[
\Gal(\overline{K}/K)=\Gal(\overline{K}(C)/K(C))\quad\text{ and }\quad \Gal(\overline{K}/K_{(v)})=\Gal(\overline{K}(C)/K_{(v)}(C))\,.
\]This induces Hochschild-Serre spectral sequences
\[
E_2^{pq}(K)=H^p(K\,,\,H^q(\overline{K}(C)\,,\,\mathbb{Q}_2/\mathbb{Z}_2(3))) \Longrightarrow H^{p+q}(K(C)\,,\,\mathbb{Q}_2/\mathbb{Z}_2(3))
\]and
\[
E_2^{pq}(K_{(v)})=H^p(K_{(v)}\,,\,H^q(\overline{K}(C)\,,\,\mathbb{Q}_2/\mathbb{Z}_2(3))) \Longrightarrow H^{p+q}(K_{(v)}(C)\,,\,\mathbb{Q}_2/\mathbb{Z}_2(3))\,.
\]Using
\[
\mathrm{cd}_2(\overline{K}(C))\le 1\,\quad \text{ and }\quad \mathrm{cd}_2(K_{(v)})\le \mathrm{cd}_2(K)\le 3\,,
\]one finds easily that the above spectral sequences induce canonical isomorphisms
\[
H^4(K(C)\,,\,\mathbb{Q}_2/\mathbb{Z}_2(3))\cong H^3(K\,,\,H^1(\overline{K}(C)\,,\,\mathbb{Q}_2/\mathbb{Z}_2(3)))
\]
and
\[
H^4(K_{(v)}(C)\,,\,\mathbb{Q}_2/\mathbb{Z}_2(3))\cong H^3(K_{(v)}\,,\,H^1(\overline{K}(C)\,,\,\mathbb{Q}_2/\mathbb{Z}_2(3)))\,.
\]Since $H^1(\overline{K}(C)\,,\,\mathbb{Q}_2/\mathbb{Z}_2(3))\cong \overline{K}(C)^*\otimes\mathbb{Q}_2/\mathbb{Z}_2(2)$, we need only prove the injectivity of the natural map
\[
H^3(K\,,\,\overline{K}(C)^*\otimes\mathbb{Q}_2/\mathbb{Z}_2(2))\lra \prod_{v\in\Omega_A}H^3(K_{(v)}\,,\,\overline{K}(C)^*\otimes\mathbb{Q}_2/\mathbb{Z}_2(2))
\]is injective.

By Lemma$\;$\ref{lemma2p1TEMP}, we have an injection
\[
H^3(K\,,\,\overline{K}(C)^*\otimes\mathbb{Q}_2/\mathbb{Z}_2(2))\hookrightarrow \bigoplus_{p\in C^{(1)}}H^3(K(P)\,,\,\mathbb{Q}_2/\mathbb{Z}_2(2))\,.
\]
For each $v$, let $C_{(v)}=C\times_KK_{(v)}$ be the base extension of $C$ and let $K_{(v)}(C)$ denote the function field of
$C_{(v)}$. By functoriality, we may reduce to proving the injectivity of the map
\[
\varphi\,:\;\;\bigoplus_{P\in C^{(1)}}H^3(K(P)\,,\,\mathbb{Q}_2/\mathbb{Z}_2(2))\lra \prod_{v\in\Omega_A}\bigoplus_{Q\in C_{(v)}^{(1)}}H^3(K_{(v)}(Q)\,,\,\mathbb{Q}_2/\mathbb{Z}_2(2))\,.
\]For fixed $v$ and $P\in C^{(1)}$, the corresponding component $\varphi_{v,\,P}$ of  the map $\varphi$ is given by
\[
\varphi_{v,\,P}\,:\;\;H^3(K(P)\,,\,\mathbb{Q}_2/\mathbb{Z}_2(2))\lra \bigoplus_{Q\,|\,P}H^3(K_{(v)}(Q)\,,\,\mathbb{Q}_2/\mathbb{Z}_2(2))\,,
\]where $Q$ runs over the points of the fiber $C_{(v)}\times_CP=\mathrm{Spec}(K_{(v)}\otimes_KK(P))$.
An element $\alpha=(\alpha_P)\in \oplus_{P}H^3(K(P)\,,\,\mathbb{Q}_2/\mathbb{Z}_2(2))$ lies in $\ker(\varphi)$ if and only if for each $P\in C^{(1)}$, $\alpha_P$
lies in the kernel of
\[
\varphi_P=\prod_{v}\varphi_{v,\,P}\,:\;\;H^3(K(P)\,,\,\mathbb{Q}_2/\mathbb{Z}_2(2))\lra\prod_{v}H^3_{\text{\'et}}(K_{(v)}\otimes_KK(P)\,,\,\mathbb{Q}_2/\mathbb{Z}_2(2))\,.
\]It suffices to prove that for every $P$, the map $\varphi_P$ is injective.

Replacing $K(P)$ by the separable closure of $K$ in $K(P)$ if necessary, we may assume that $K(P)/K$  is a finite separable extension.
Then we have
\[
K_{(v)}\otimes_KK(P)\cong\prod_{w\,|\,v} K(P)_{(w)}\,,
\](cf. \cite[IV.18.6.8]{EGA8}). So the map $\varphi_P$ gets identified with the natural map
\[
H^3(K(P)\,,\,\mathbb{Q}_2/\mathbb{Z}_2(2))\lra \prod_{w}H^3(K(P)_{(w)}\,,\,\mathbb{Q}_2/\mathbb{Z}_2(2))\,,
\]where $w$ runs over divisorial valuations of $K(P)$. This map is injective by Theorem$\;$\ref{thm1p3temp}.
The theorem is thus proved.
\end{proof}

\begin{coro}\label{coro2p3TEMP}
Let $K$ be the function field of a $p$-adic arithmetic surface or a local henselian surface with finite residue field of charachteristic
$p$. Assume $p\neq 2$ in the local henselian case. Let $C$ a smooth projective conic in $\mathbb{P}^2_K$.

Then the natural map
\[
I^4(K(C))\lra \prod_{v\in\Omega_A} I^4(K_v(C))
\]is injective, where $v$ runs over  the set $\Omega_A$ of divisorial valuations of $K$.
\end{coro}
\begin{proof}
For $F=K(C)$ or $K_v(C)$, we have $\mathrm{cd}_2(F)\le 4$. By the degree 4 case of the Milnor conjecture (cf. \cite{Voe03} and \cite{OVV}), we have an isomorphism $I^4(F)\cong H^4(F,\,\mathbb{Z}/2)$. (In the $p$-adic arithmetic case, we can also deduce this isomorphism from
\cite[p.655, Prop.$\;$2]{AEJ86} together with \cite[Thm.$\;$3.4]{Le10}.) The result then follows immediately from Theorem$\;$\ref{thm2p2TEMP}.
\end{proof}

\subsection{Case of even index}

\begin{prop}\label{prop3p2TEMP}
Let  $K$ be the function field of a $p$-adic arithmetic surface or a local henselian surface with finite residue field of characteristic $p$. Let $L/K$ be a quadratic field extension, $(D,\,\tau)$ a central division algebra over $L$ with a unitary $L/K$-involution whose index is not divisible by $4$. Let $h$ be a nonsingular hermitian form over $(D,\,\tau)$ which has even rank, trivial discriminant and trivial Rost invariant $($cf. $(\ref{para3p1TEMP}))$.
Assume $p\neq 2$ if $\ind(D)$ is even. In the local henselian case, assume further that the Hasse principle with respect to divisorial valuations holds for quadratic forms of rank $6$ over $K$.

Then we have $[h]=0\in W(D,\,\tau)$  if and only if $[h\otimes_KK_v]=0\in W(D\otimes_KK_v\,,\,\tau)$ for every $v\in\Omega_A$.
\end{prop}
\begin{proof}If the index $\mathrm{ind}(D)$ is odd, the result is already proved in Proposition$\;$\ref{prop5p1temp}. We assume next that $\ind(D)$
is even and not divisible by 4.

We first consider the case where $D$ is a quaternion algebra.
As in (\ref{para6p4TEMP}), we write $D=D_0\otimes_KL$ with $D_0$ a quaternion division algebra over $K$ and $L=K(\sqrt{d})$ with
$d\in K^*$, and we have Suresh's exact sequence
\begin{equation}\label{eq3p2p1TEMP}
W(L)\overset{\wt{\pi}_1}{\lra}W(D_0\,,\,\tau_0)\overset{\wt{\rho}}{\lra}W(D\,,\,\tau)\overset{p_2}{\lra}W^{-1}(D_0,\,\tau_0)
\end{equation}

Let $C\subseteq \mathbb{P}^2_K$ be the smooth projective conic associated to the quaternion algebra $D_0$. Then the algebra $D\otimes_KK(C)=D_0\otimes_KL(C)$ is a split central simple algebra over $L(C)$ with a unitary $L(C)/K(C)$-involution $\tau$. By Morita theory, the hermitian form $h\otimes_KK(C)$ over $(D\otimes_KK(C)\,,\,\tau)$
corresponds to a hermitian form $h'_C$ over $(L(C)\,,\,\iota)$, where $\iota$ denotes the nontrivial element of the Galois group $\mathrm{Gal}(L(C)/K(C))$.
The trace form $q_{h,\,C}$ of $h'_C$ gives a quadratic form over $K(C)$. By \cite[Example$\;$31.44]{KMRT}, the quadratic form $q_{h,\,C}$ has even rank,
trivial discriminant, trivial Clifford invariant and trivial Rost invariant, since $h'_C$ has even rank, trivial discriminant and trivial Rost invariant (these invariants being invariant under Morita equivalence). Hence in the Witt group $W(K(C))$ we have $[q_{h,\,C}]\in I^4(K(C))$. Since $h$ is locally hyperbolic, it follows from Corollary$\;$\ref{coro2p3TEMP} that $[q_{h,\,C}]=0\in W(K(C))$, whence $[h\otimes_KK(C)]=0\in W(D\otimes_KK(C)\,,\,\tau)$. In the commutative diagram
\[
\begin{CD}
W(D,\,\tau) @>p_2>> W^{-1}(D_0,\,\tau_0)\\
@VVV @VVV\\
W(D\otimes_KK(C)\,,\,\tau) @>p_2>>  W^{-1}(D_0\otimes_KK(C)\,,\,\tau_0)
\end{CD}
\]the right vertical map is injective by \cite{PSS}. So we have $p_2(h)=0\in W^{-1}(D_0,\,\tau_0)$. The exactness of the sequence
\eqref{eq3p2p1TEMP} implies that $[h]=\tilde{\rho}([h_0])$ for some hermitian form $h_0$ over $(D_0,\,\tau_0)$ of even rank.

Let $\lambda=\mathrm{Pf}(h_0)\in K^*/\Nrd(D_0^*)$ be the pfaffian norm of $h_0$. Since $h$ is locally hyperbolic, by considering Suresh's exact sequence locally, we see that $(h_0)_v$ lies in the image of $(\widetilde{\pi}_1)_v$ for every $v$. By Lemma$\;$\ref{lemma1TEMP}, this implies that $\lambda\in \Nrd((D_0)^*_v).N_{L_v/K_v}(L_v^*)$ for every $v$. In other words, the quadratic form
\[
\phi:=\lambda.n_{D_0}-\langle 1\,,\,-d \rangle\,,
\]where $n_{D_0}$ denotes the norm form of the quaternion algebra $D_0$, is isotropic over every $K_v$.
By the assumption on the Hasse principle for quadratic forms of rank 6 (and \cite[Thm.$\;$3.1]{CTPaSu} in the $p$-adic arithmetic case), $\phi$ is isotropic over $K$, which shows
$\lambda\in \Nrd(D_0^*).N_{L/K}(L^*)$. As was mentioned in the proof of Corollary$\;$\ref{coro3p4temp}, the field $K$ has $u$-invariant 8. So by Lemma$\;$\ref{lemma3TEMP}, we have $[h_0]\in \mathrm{Im}(\widetilde{\pi}_1)$.
Hence $[h]=\tilde{\rho}([h_0])=0\in W(D,\,\tau)$ as desired.

Consider next the general case where $\ind(D)$ is even and not divisible by 4. In this case we have $D=Q\otimes_LD'$ for some quaternion division algebra $Q$
over $L$ and some central division algebra $D'$ of odd index over $L$. By \cite[Lemma$\;$3.3.1]{BP1}, there is an odd degree separable extension $K'/K$ such that $D'\otimes_KK'=D'\otimes_LLK'$ is split. By Morita theory, there is a unitary $LK'/K'$-involution $\sigma$ on $H\otimes_LLK'$ and a hermitian form $f$
over $(H\otimes_LLK'\,,\,\sigma)$ such that the involution $\tau$ on $D\otimes_LLK'$ is adjoint to $f$, and moreover, the form $h_{K'}$ over $(D\otimes_LLK'\,,\,\tau)$ corresponds to a hermitian form $h'$ over $(H\otimes_LLK'\,,\,\sigma)$, which has even rank, trivial discriminant and trivial Rost invariant. The hypothesis that $h$ is locally hyperbolic over every $K_v$ implies that $h'$ is locally hyperbolic over every $K'_w$, where $w$ runs over the set of divisorial valuations of $K'$. By the previous case, $[h']=0\in W(H\otimes_LLK'\,,\,\sigma)$ and hence $[h]=0\in W(D\otimes_LLK'\,,\,\tau)$. Since the degree
$[LK': L]=[K': K]$ is odd, the natural map $W(D,\,\tau)\to W(D\otimes_LLK'\,,\,\tau)$ is injective by a theorem of Bayer-Fluckiger and Lenstra (cf. \cite[p.80,
Coro.$\;$6.18]{KMRT}). So we get
$[h]=0\in W(D,\,\tau)$. This completes the proof.
\end{proof}

\begin{thm}\label{thm3p3TEMP}
Let $K$ be the function field of a $p$-adic arithmetic surface or a local henselian surface with finite residue field of characteristic $p$. Let $L/K$ be a quadratic field extension, $(D,\,\tau)$ a central division algebra over $L$ with a unitary $L/K$-involution whose index $\mathrm{ind}(D)$ is square-free. Let $h$ be a nonsingular hermitian form over $(D,\,\tau)$.

Assume $p\neq 2$ if $\ind(D)$ is even. In the local henselian case, assume further that  $p\nmid \mathrm{ind}(D)$ and that the Hasse principle with respect to divisorial valuations holds for quadratic forms of rank $6$ over $K$.

Then the natural map
\[
H^1(K\,,\,\bfSU(h))\lra \prod_{v\in \Omega_A}H^1(K_v\,,\,\bfSU(h))
\]has trivial kernel.
\end{thm}
\begin{proof}
Let $\xi\in H^1(K\,,\,\bfSU(h))$ be a class that is locally trivial. Let the image of $\xi$ in $H^1(K,\,\bfU(h))$ correspond to a hermitian form
$h'$. The form $h'\bot (-h)$ has even rank, trivial discriminant and is locally hyperbolic. We claim that the Rost invariant $\mathscr{R}(h'\bot(-h))$
is trivial. Indeed, as $\xi$ is locally trivial, $R_{\bfSU(h)}(\xi)$ is locally trivial in $H^3(K_v,\,\mathbb{Q}/\mathbb{Z}(2))$ for every $v$. By Theorem$\;$\ref{thm1p3temp},
$R_{\bfSU(h)}(\xi)=0$. There is a group homomorphism
\[
\bfSU(h)\lra \bfSU(h\bot(-h))\,,\;\; f\longmapsto (f,\,\mathrm{id})
\]which induces a map
\[
\alpha\,:\;\;H^1(K,\,\bfSU(h))\lra H^1(K\,,\,\bfSU(h\bot(-h)))\,.
\]The image $\al(\xi)$ of $\xi$ lifts the class $[h'\bot (-h)]\in H^1(K,\,\bfU(h\bot(-h)))$.
By general property of the (usual) Rost invariant, there is an integer $n_{\al}$ such that
\[
R_{\bfSU(h\bot(-h))}(\al(\xi))=n_{\al}R_{\bfSU(h)}(\xi)\,.
\]We have thus $\mathscr{R}(h'\bot(-h))=R_{\bfSU(h\bot(-h))}(\al(\xi))=0$ since $\xi$ has trivial Rost invariant.
Now Proposition$\;$\ref{prop3p2TEMP} implies that the two forms $h', h$ over $(D,\,\tau)$ are isomorphic.

Consider the cohomology exact sequence
\begin{equation}\label{eq3p3p1TEMP}
1\lra \frac{R^1_{L/K}\mathbb{G}_m(K)}{\Nrd(\bfU(h)(K))}\overset{\varphi}{\lra}H^1(K,\,\bfSU(h))\lra H^1(K,\,\bfU(h))
\end{equation}arising from the exact sequence of algebraic groups
\[
1\lra \bfSU(h)\lra \bfU(h)\overset{\Nrd}{\lra}R^1_{L/K}\mathbb{G}_m\lra 1\,.
\]The fact that $h'\cong h$ implies that $\xi$ lies in the image of the map $\varphi$ in the above cohomology exact sequence \eqref{eq3p3p1TEMP}.
Considering the sequence \eqref{eq3p3p1TEMP} locally and using Lemma$\;$\ref{lemma5p2temp}, we conclude that
$\xi$ is trivial in $H^1(K,\,\bfSU(h))$, thus proving the theorem.
\end{proof}

\noindent \emph{Acknowledgements.} The author thanks Prof.\! Jean-Louis Colliot-Th\'{e}l\`{e}ne for helpful discussions.

\addcontentsline{toc}{section}{

\clearpage \thispagestyle{empty}

\end{document}